\documentclass[10pt]{amsart}
\usepackage{amssymb, amsthm, amsmath}
\usepackage[all]{xy}
\usepackage{graphicx}
\usepackage{psfrag}

\newtheorem{prop}{Proposition}
\newtheorem{thm}{Theorem}
\newtheorem{cor}{Corollary}
\newtheorem{lemma}{Lemma}
\theoremstyle{definition}
\newtheorem{defn}{Definition}
\newtheorem{example}{Example}

\newcommand\C{{\mathbb C}}

\newcommand\N{{\mathbb N}}
\newcommand\F{{\mathrm F}}

\newcommand{\ti}{\vartheta}
\newcommand{\Ti}{\Theta}

\newcommand\cW{{\mathcal W}}

\newcommand\X{{\mathfrak X}}
\newcommand\Z{{\mathbb Z}}

\newcommand\AS{{\mathfrak S}}
\newcommand\BS{{\mathfrak B}}
\newcommand\CS{{\mathfrak C}}
\newcommand\DS{{\mathfrak D}}
\newcommand\XX{{\mathrm X}}
\newcommand\YY{{\mathrm Y}}
\newcommand\J{{\mathrm J}}

\newcommand\fraka{{\mathfrak a}}
\newcommand\frakb{{\mathfrak b}}
\newcommand\dd{{\mathfrak d}}

\newcommand\al{\alpha}

\newcommand\la{\lambda}
\newcommand\s{{\sigma}}
\newcommand\x{{\mathrm{x}}}
\newcommand\y{{\mathrm{y}}}
\newcommand\om{{\varpi}}

\newcommand\ssm{\smallsetminus}

\newcommand\noin{\noindent}

\newcommand\bull{{\scriptscriptstyle \bullet}}
\newcommand\eqto{\stackrel{\lower1.5pt\hbox{$\scriptstyle\sim\,$}}\to}
\newcommand\ov{\overline}
\newcommand\hra{\hookrightarrow}

\newcommand\wh{\widehat}
\newcommand\wt{\widetilde}

\DeclareMathOperator{\Sp}{Sp} 
\DeclareMathOperator{\SL}{SL} \DeclareMathOperator{\GL}{GL}
\DeclareMathOperator{\LG}{LG} \DeclareMathOperator{\IG}{IG}
 \DeclareMathOperator{\OG}{OG}
 
\DeclareMathOperator{\HH}{\mathrm{H}} 
\DeclareMathOperator{\type}{\mathrm{type}}
\DeclareMathOperator{\rank}{\mathrm{rank}}

\newcommand{\ignore}[1]{}

\psfrag{k}{$k$}
\psfrag{(r,c)}{$[r,c]$}
\psfrag{(r',c')}{$[r',c']$}

\begin{document}

\title[A Giambelli formula for classical $G/P$ spaces]
{A Giambelli formula for classical $G/P$ spaces}

\date{March 30, 2014}

\author{Harry~Tamvakis} \address{University of Maryland, Department of
Mathematics, 1301 Mathematics Building, College Park, MD 20742, USA}
\email{harryt@math.umd.edu}


\thanks{The author was supported in part by NSF Grant DMS-0901341.}

\begin{abstract}
Let $G$ be a classical complex Lie group, $P$ any parabolic subgroup
of $G$, and $G/P$ the corresponding partial flag variety. We prove an
explicit combinatorial Giambelli formula which expresses an arbitrary
Schubert class in $\HH^*(G/P)$ as a polynomial in certain special
Schubert class generators. Our formula extends to one that applies to
the torus-equivariant cohomology ring of $G/P$ and to the setting of
symplectic and orthogonal degeneracy loci.
\end{abstract}

\maketitle

\setcounter{section}{-1}

\section{Introduction}

The Giambelli formula \cite{G} is one of the fundamental
results concerning Schubert calculus in the cohomology ring of the
Grassmannian $\X$.  The variety $\X$ has a decomposition into
Schubert cells, which gives an additive basis of Schubert classes for
the cohomology of $\X$. On the other hand, the ring $\HH^*(\X,\Z)$ is
generated by certain {\em special} Schubert classes, which are the
Chern classes of the universal quotient bundle over $\X$. The formula
of Giambelli expresses a general Schubert class as a determinant of a
Jacobi-Trudi matrix with entries given by special classes. One can
show that this formula is equivalent to the Pieri rule \cite{P}; 
see for instance \cite{T4}.

The Schubert calculus on $\X$ can be generalized to any homogeneous
space $G/P$, where $G$ is a complex reductive Lie group and $P$ a
parabolic subgroup of $G$. However, more than a century since the
theorems of Pieri and Giambelli were discovered, no combinatorially
explicit analogues of these results are known in this generality,
unless the Lie group $G$ is of  type A. One reason for this is that
there is no uniform way to extend the notion of a special Schubert
class over all possible Lie types and parabolics. Another serious 
concern is the more difficult algebro-combinatorial questions that 
arise in the other Lie types, about which more below. 

When $G$ is a {\em classical} Lie group, one can define special
Schubert class generators for the cohomology ring $\HH^*(G/P)$
uniformly, as follows. In this situation, the variety $G/P$
parametrizes partial flags of subspaces of a vector space, which in
types B, C, and D are required to be isotropic with respect to an
orthogonal or symplectic form. First, the special Schubert varieties
on any Grassmannian are defined as the locus of (isotropic) linear
subspaces which meet a given (isotropic or coisotropic) linear
subspace nontrivially, following \cite{BKT1}.  The special Schubert
classes are the cohomology classes determined by these Schubert
varieties.  Finally, the special Schubert classes on a partial flag
variety $G/P$ are the pullbacks of special Schubert classes on
Grassmannians, in agreement with the convention in type A. In most
cases, these special classes are equal to the Chern classes of the
universal quotient bundles over $G/P$, up to a factor of two.  The
{\em Giambelli problem} then is to find an explicit combinatorial
formula which writes a general Schubert class in $\HH^*(G/P)$ as a
polynomial in the special classes. One of our motivations for this is
the fact that the known Giambelli formulas expressed in terms of the
above special Schubert classes have straightforward -- often identical
-- extensions to the small quantum cohomology ring of $G/P$; see
\cite{Bertram, CF, FGP, KTlg, KTog, BKT3}.

The modern formulation of the Giambelli problem is in the setting of
an algebraic family of varying partial flag varieties, with
applications to {\em degeneracy loci} of vector bundles. This story
also has a long history, from the work of Thom-Porteous and
Kempf-Laksov \cite{KL} to the generalizations by Fulton, Pragacz, and
others \cite{Fu1, Fu2, Fu3, Fu4, P1, PR2, LP1, KT, IMN}.  In type A,
the project culminated with the combinatorial understanding of the
polynomials representing quiver loci \cite{BF, BKTY, KMS}. Let $T$
denote a maximal torus and $B$ a Borel subgroup of $G$ with $T\subset
B$. Graham \cite{Gra} observed that the degeneracy locus problem for
the classical groups from \cite{Fu1, Fu2, Fu3} is equivalent to the
Giambelli problem for the Schubert classes in the {\em $T$-equivariant
  cohomology ring} of $G/B$.

The degeneracy locus formulas we obtain here have a similar shape for
all the classical groups, and solve the Giambelli problem for the
$T$-equivariant cohomology of any classical $G/P$ space (when $G$ is
an even orthogonal group, the problem is reduced to the main theorem
of \cite{BKT4}). It should be noted that in type D there are some
unavoidable differences due to the presence of the Euler class
\cite{EGr}. For instance, the Chern classes of the universal vector
bundles over a non-maximal even orthogonal Grassmannian do not
generate its cohomology ring, even with real coefficients; see
\cite{T1, BKT1, BKT4} for a detailed analysis.

The case of the complete flag variety $G/B$, where the cohomology is
generated by Schubert divisors, is more amenable to study.
Bernstein-Gelfand-Gelfand \cite{BGG} and Demazure \cite{D1, D2} used
divided difference operators to construct an algorithm that produces
polynomials which represent the Schubert classes on $G/B$ in the Borel
presentation \cite{Bo} of the cohomology ring.  For the general linear
group, Lascoux and Sch\"utzenberger \cite{LS1} applied this method to
define {\em Schubert polynomials}, a particularly nice choice of
representatives with rich combinatorial properties. A positive
combinatorial formula for the coefficients of Schubert polynomials was
given by Billey, Jockusch, and Stanley \cite{BJS} -- thus resolving
the Giambelli question in the case of $\GL_n/B$. There remained
combinatorial difficulties to generalize this to a Giambelli formula
which holds on any type A partial flag variety. The answer for
$\GL_n/P$ was obtained by Buch, Kresch, Yong, and the author \cite[\S
  5]{BKTY} in the course of their work on the quiver formulas of Buch
and Fulton \cite{BF}.

Pragacz \cite{P2} solved the Giambelli problem for maximal isotropic
orthogonal and symplectic Grassmannians using a Schur Pfaffian
\cite{Sch}. At the other extreme, for the full flag varieties $G/B$ in
types B, C, and D, a family of Schubert polynomials analogous to the
one in type A is not uniquely determined (see \cite{FK2} for a
discussion of this phenomenon and \cite{BH, Fu2, LP1, LP2, T2,
T3} for examples) and the combinatorics is more challenging. One
reason for this is that the strong {\em stability property} of type A
Schubert polynomials under the natural inclusions of the Weyl groups
must be understood differently in types B, C, and D. Another is that
the images of the special Schubert classes in the stable cohomology
rings of isotropic Grassmannians are not algebraically independent.

Giambelli formulas for non-maximal isotropic Grassmannians were
discovered only recently \cite{BKT2, BKT4}. Unlike most previously
known examples, the nature of these formulas is not determinantal --
instead, they are expressed using Young's raising operators
\cite{Y}. The corresponding Giambelli polynomials are the {\em theta}
and {\em eta} polynomials, which play the same role as the Schur
polynomials do for the type A Grassmannian. This theory is essential
for further progress and will be generalized in the present article.
The results provide a combinatorial link between the classical $G/P$
Giambelli problem and the quiver formulas of \cite{BF, BKTY}.

We now state one of our main theorems, referring to \S
\ref{prelims}--\S \ref{sdl} for the precise definitions.  Equip the
vector space $E=\C^{2n}$ with a nondegenerate skew-symmetric bilinear
form.  Fix a sequence $\dd \, :\, d_1 < \cdots < d_p$ of positive
integers with $d_p \leq n$.  Let $\X(\dd)$ be the variety
parametrizing partial flags of subspaces
\begin{equation}
\label{Eseq}
E_\bull \ :\ 0 \subset E_1 \subset \cdots \subset
E_p \subset E
\end{equation} 
with $\dim E_i = d_i$ for each $i$ and $E_p$ isotropic. The Schubert
varieties $\X_w$ in $\X(\dd)$ and their cohomology classes $[\X_w]$
are indexed by signed permutations $w$ in the Weyl group of $\Sp_{2n}$
whose descent positions are included among the $d_i$. Let $E'_j =
E/E_{p+1-j}$ for $1\leq j\leq p$; the Chern classes of the
tautological quotient bundles $E'_j$ are the special Schubert classes
in $\HH^*(\X(\dd),\Z)$. We then have
\begin{equation}
\label{introgiam}
[\X_w] = \sum_{\underline{\la}} 
e^w_{\underline{\la}}\,
\Ti_{\la^1}(E'_1)s_{\la^2}(E'_2-E'_1)\cdots s_{\la^p}(E'_p-E'_{p-1})
\end{equation}
summed over all sequences of partitions
$\underline{\la}=(\la^1,\ldots,\la^p)$ with $\la^1$ $k$-strict, where
$k=n-d_p$. Here $\Ti_{\la^1}$ and $s_{\la^i}$ denote theta and Schur
polynomials, respectively, and the coefficient $e^w_{\underline{\la}}$ is
a nonnegative integer which counts the number of $p$-tuples of leaves of
shape $\underline{\la}$ in the groves of the transition forest
associated to $\dd$ and $w$. 

Let $\IG(n-k,2n)$ denote the Grassmannian parametrizing isotropic
linear subspaces of $E$ of dimension $n-k$. The morphism which sends
$E_\bull$ to $E_p$ realizes $\X(\dd)$ as a fiber bundle over
$\IG(n-k,2n)$ with fiber equal to a type A partial flag variety.  The
mixed nature of the ingredients in formula (\ref{introgiam}) is in
harmony with this fact. For the type A and orthogonal partial flag
varieties, the Giambelli formulas and their underlying combinatorics
and geometry are entirely analogous to (\ref{introgiam}). The
definition of the transition forest differs slightly between the
types, and the role of the theta polynomial is played by a Schur or 
an eta polynomial, respectively.

Our proof of (\ref{introgiam}) and the corresponding formulas for
degeneracy loci is mainly combinatorial. We work with the Schubert
polynomials of Billey and Haiman \cite{BH}, and more generally with
their double versions introduced by Ikeda, Mihalcea, and Naruse
\cite{IMN}. These objects have most of the combinatorial properties of
the type A Schubert polynomials, which are crucial in order for us to
obtain explicit positive expressions such as (\ref{introgiam});
however their connection with the geometry is less clear. This latter
problem was recently addressed for the single Schubert polynomials in
\cite{T2, T3} and in \cite{IMN} for their double counterparts. Using
this theory, our geometric results follow readily by combining the
Giambelli formulas for isotropic Grassmannians from \cite{BKT2, BKT4}
with new {\em splitting theorems} for these Schubert polynomials,
which are analogues of \cite[Thm.\ 4 and Cor.\ 3]{BKTY} for the other
classical Lie types.

A first step towards splitting the Billey-Haiman Schubert polynomials
was formulated by Yong \cite{Yo}. However, his result -- and most
previous work on these polynomials -- does not suffice for the
aforementioned applications to geometry. The point is that the flag of
subspaces $E_\bull$ in (\ref{Eseq}) need not contain a Lagrangian
(i.e., a maximal isotropic) subspace. In terms of the Weyl group, this
is simply the fact that the first descent position $k$ of a signed
permutation $w$ need not be at {\em zero}. When $k=0$, the analysis is
easier because the cohomology ring of the Lagrangian Grassmannian
$\LG(n,2n)$ is more accessible, a result that goes back to Ehresmann
\cite{Eh}. The underlying reason for this is that the Schubert classes
on $\LG$ are indexed by {\em fully commutative elements} of the Weyl
group \cite{St2}. As a consequence, the Schubert calculus on $\LG$ is 
very similar to the classical one.

To complete the algebro-combinatorial picture for any $G/P$ space, we
need to extend the established theory to include the $k$-Grassmannian
elements which are not fully commutative, a program initiated in
\cite{BKT2, T4}. For our purposes here we introduce the {\em mixed
Stanley functions}, denoted by $J_w(X\,;Y)$ in type C.  The function
$J_w$ is defined as a sum of products of type A and type C Stanley
symmetric functions in two distinct sets of variables. For each fixed
$m\geq 0$, it includes among its coefficients the number of reduced
words of the signed permutation $w$ such that the last $m$ letters in
the word are positive (the letter $0$ is used to denote the sign
change). If $w$ has no descent positions less than $k$, then $J_w$ --
suitably restricted -- is a nonnegative integer linear combination of
theta polynomials indexed by $k$-strict partitions $\la$. The case
when $k=0$ was studied earlier by several authors \cite{H, Kr, L1,
B}. It is the {\em mixed Stanley coefficients} that appear in this
expansion which enter into the splitting and degeneracy locus
formulas.

Finally, our combinatorial interpretation of the numbers
$e^w_{\underline{\la}}$ in (\ref{introgiam}) depends on the {\em
transition equations} of Lascoux and Sch\"utzenberger \cite{LS3} and
Billey \cite{B}. More precisely, we define in \S \ref{tes} and \S
\ref{ogps} analogues of the Lascoux-Sch\"utzenberger transition tree
\cite{LS3}, which are rooted at suitable elements in the Weyl groups
of the symplectic and orthogonal groups. We remark that the difference
between the stability property of type A Schubert polynomials and the
analogous property in the other classical Lie types is also reflected
in the construction of these trees -- namely, there is a certain
branching rule in \cite{LS3} which has no counterpart in types B, C,
and D. It would be interesting to have tableau formulas for the
$e^w_{\underline{\la}}$, similar to the ones in \cite{BF, BKTY, KMS}
for quiver coefficients.

This article is organized so that most of the exposition is in type C;
a final section explains the analogous picture in types B and D, and
we refer to \cite[\S 4,5]{BKTY} for the results in type A. We review
the Schur, theta, and Schubert polynomials, as well as the Stanley
symmetric functions we require in \S \ref{prelims}. The mixed Stanley
functions and their basic properties are studied in \S
\ref{transition}; this includes applications to enumerating reduced
words and combinatorial rules for the product of two theta
polynomials. In particular we give a short proof of Proposition
\ref{prodprop}, which includes as a special case the Pieri type
products studied by Pragacz and Ratajski \cite{PR1} on non-maximal
isotropic Grassmannians. Our splitting theorems for Schubert
polynomials are proved in \S \ref{splitC}, and the applications to
symplectic degeneracy loci and Giambelli formulas are deduced in \S
\ref{sdl} and \S \ref{flags}, respectively.  We conclude with the
Schubert splitting results for the orthogonal groups in \S \ref{ogps}.

I grateful to my collaborators Anders Buch, Andrew Kresch, and
Alexander Yong for their hard work on related papers, especially
\cite{BKTY} and \cite{BKT1, BKT2}. I thank Leonardo Mihalcea for
conversations about his joint paper \cite{IMN} on double Schubert
polynomials, which played an important role in this project.  I also
thank the referees for their comments which helped to improve the
exposition.

\section{Preliminaries}
\label{prelims}

\subsection{Schur and theta polynomials}

An {\em integer sequence} is a sequence of integers $\al=(\al_1,
\al_2,\ldots)$ only finitely many of which are non-zero. The largest
integer $\ell\geq 0$ such that $\al_\ell\neq 0$ is called the {\em
length} of $\al$, denoted $\ell(\al)$; we will identify an integer
sequence of length $\ell$ with the vector consisting of its first
$\ell$ terms, and set $|\al| = \sum \al_i$. An integer sequence $\la$
is a {\em partition} if $\la_i \geq \la_{i+1}\geq 0$ for all $i$. We
will represent partitions $\la$ by their Young diagram of boxes, let
$\la'$ denote the conjugate (or transpose) of $\la$, and write
$\mu\subset\la$ for the containment relation between two Young
diagrams.  Following Young \cite{Y}, given any integer sequence
$\alpha$ and natural numbers $i<j$, we define
\[
R_{ij}(\alpha) = (\alpha_1,\ldots,\alpha_i+1,\ldots,\alpha_j-1,
\ldots).
\] 
A raising operator $R$ is any monomial in these $R_{ij}$'s.  If
$(u_1,u_2,\ldots)$ is any ordered set of commuting independent
variables and $\al$ is an integer sequence, we let $u_{\al} =
\prod_iu_{\al_i}$, with the understanding that $u_0=1$ and $u_r = 0$
if $r<0$. For any raising operator $R$, let $R\,u_{\al} = u_{R\al}$.

Let $c=(c_1,c_2,\ldots)$ and $d=(d_1,d_2,\ldots)$ be two families
of commuting variables. Define elements $h_r$ for $r\in \Z$ 
by the identity of formal power series
\[
\sum_{r=-\infty}^{+\infty}h_rt^r = 
\left(\sum_{i=0}^\infty(-1)^ic_it^i\right)^{-1}
\left(\sum_{i=0}^\infty(-1)^id_it^i\right).
\]
Consider the raising operator expression
\[
R^{0}= \prod_{i<j}(1-R_{ij})
\]
and for any partition $\la$, define the {\em Schur
polynomial} $s_\la(c-d)$ by
\[
s_\la(c-d) = R^{0}\, h_\la = \det(h_{\la_i+j-i})_{i,j}.
\]
Let $Y=(y_1,y_2,\ldots)$ and $Z=(z_1,z_2,\ldots)$ be two infinite sets
of variables, and define the elementary symmetric functions $e_r(Y)$
by the generating function
\[
\prod_{i=1}^{\infty}(1+y_it)=\sum_{r=0}^{\infty}e_r(Y)t^r.
\]
The {\em supersymmetric Schur function} $s_\la(Y/Z)$ is obtained from
$s_\la(c-d)$ by setting $c_r = e_r(Y)$ and $d_r = e_r(Z)$ for all
$r\geq 1$. The usual Schur $S$-functions satisfy the identities
$s_\la(Y) = s_\la(Y/Z)\vert_{Z=0}$ and $s_\la(0/Z) =
s_\la(Y/Z)\vert_{Y=0} = (-1)^{|\la|}s_{\la'}(Z)$. In particular, for
each integer $r$, the function $s_r(Y)$ is the complete symmetric function
in the variables $Y$, also denoted $h_r(Y)$.

Fix an integer $k\geq 0$. A partition $\la$ is {\em $k$-strict} if all 
its parts $\la_i$ greater than $k$ are distinct; $\la$ is called 
{\em strict} if it is $0$-strict. Any such $\la$ determines
a raising operator expression $R^\la$ by the prescription
\[
R^\la = \prod_{i<j}(1-R_{ij})\prod_{{i<j}\atop{\la_i+\la_j > 2k+j-i}}
(1+R_{ij})^{-1}
\]
where the second product is over all pairs $i<j$ such that
$\la_i+\la_j > 2k + j-i$. 
Define elements $g_r$ for $r\in \Z$ by the identity 
\[
\sum_{r=-\infty}^{+\infty}g_rt^r = 
\left(\sum_{i=0}^\infty c_it^i\right)
\left(\sum_{i=0}^\infty d_it^i\right)^{-1}
\]
and the {\em theta polynomial $\Ti_\la(c-d)$} by
\begin{equation}
\label{Tidef}
\Theta_\la(c-d) = R^\la\, g_\la. 
\end{equation}
If $X=(x_1,x_2,\ldots)$ is another infinite set of variables, the
formal power series $\ti_r(X\,;Y)$ for $r\in \Z$ are defined by the
equation
\[
\prod_{i=1}^{\infty}\frac{1+x_it}{1-x_it} \prod_{j=1}^k
(1+y_jt)= \sum_{r=0}^{\infty}\ti_r(X\,;Y)t^r.
\]
Following \cite{BKT2}, we then set $\Ti_\la(X\,;Y) = R^\la\,
\ti_\la$. The $\Ti_\la$ for $\la$ $k$-strict form a $\Z$-basis for the
ring $\Gamma^{(k)} = \Z[\ti_1,\ti_2,\ldots]$. The ring
$\Gamma=\Gamma^{(0)}$ is the ring of Schur $Q$-functions (see
\cite[III.8]{M2} and \cite{Sch}), and in this case $\ti_r(X)$ and
$\Ti_\la(X)$ are denoted by $q_r(X)$ and $Q_\la(X)$, respectively.

\subsection{The hyperoctahedral group}
\label{hypgp}

Let $W_n$ denote the hyperoctahedral group of signed permutations on
the set $\{1,\ldots,n\}$. We will adopt the notation where a bar is
written over an entry with a negative sign.  The group $W_n$ is the
Weyl group for the root system $\text{B}_n$ or $\text{C}_n$, and is
generated by the simple transpositions $s_i=(i,i+1)$ for $1\leq i \leq
n-1$ and the sign change $s_0(1)=\ov{1}$. There is a natural embedding
$W_n\hookrightarrow W_{n+1}$ defined by adjoining the fixed point
$n+1$.  The symmetric group $S_n$ is the subgroup of $W_n$ generated
by the $s_i$ for $1\leq i \leq n-1$, and is the Weyl group for the
root system $\text{A}_{n-1}$. We let $S_\infty = \cup_nS_n$ and
$W_\infty=\cup_n W_n$.

A {\em reduced word} for $w\in W_n$ is a sequence $a_1\cdots a_r$ of
elements in $\{0,1,\ldots,n-1\}$ such that $w=s_{a_1}\cdots s_{a_r}$
and $r$ is minimal (so equal to the length $\ell(w)$ of $w$). Given
any $u_1,\ldots,u_p,w\in W_\infty$, we write $u_1\cdots u_p=w$ if
$\ell(u_1)+\cdots + \ell(u_p)=\ell(w)$ and the product of
$u_1,\ldots,u_p$ is equal to $w$. In this case we say that $u_1\cdots
u_p$ is a {\em reduced factorization} of $w$.  We say that $w$ has
a {\em descent} at position $r\geq 0$ if $w_r>w_{r+1}$, where by
definition $w_0=0$.  An element $w\in W_\infty$ is {\em compatible}
with the sequence $\fraka\, :\, a_1 < \cdots < a_p$ of nonnegative
integers if all descent positions of $w$ are contained in $\fraka$.
For such $w$, we say that a reduced factorization $u_1\cdots u_p=w$ is
{\em compatible with} $\fraka$ if $u_j(i)=i$ for all $j>1$ and $i \leq
a_{j-1}$.

We say that a signed permutation $w\in W_\infty$ is {\em increasing up
  to $k$} if it has no descents less than $k$.  This condition is
vacuous if $k=0$, and for positive $k$ it means that $0 < w_1 < w_2 <
\cdots < w_k$. An important special case is the {\em $k$-Grassmannian}
elements, which by definition satisfy $\ell(ws_i)=\ell(w)+1$ for all
$i\neq k$. There is a natural bijection between $k$-Grassmannian
elements of $W_\infty$ and $k$-strict partitions, obtained as
follows. If $w\in W_n$ is $k$-Grassmannian, there exist unique strict
partitions $u,\zeta,v$ of lengths $k$, $r$, and $n-k-r$, respectively,
so that
\[
w=(u_k,\ldots,u_1,
\ov{\zeta}_1,\ldots,\ov{\zeta}_r,v_{n-k-r},\ldots,v_1).
\]
Define $\mu_i$ for $1\leq i \leq k$ by
\[
\mu_i=u_i+i-k-1+\#\{j\ |\ \zeta_j > u_i\}.
\]
Then $w$ corresponds to the $k$-strict partition $\la$ such that the
lengths of the first $k$ columns of $\la$ are given
by $\mu_1,\ldots,\mu_k$, and the part of $\la$ in columns $k+1$ and
higher is given by $\zeta$. Conversely, for any $k$-strict $\la$,
the corresponding $k$-Grassmannian element is denoted by $w_\la$.

\subsection{Schubert polynomials and Stanley symmetric functions}
\label{sps}

Following \cite{FS} and \cite{FK1, FK2}, we will use the nilCoxeter
algebra $\cW_n$ of the hyperoctahedral group $W_n$ to define Schubert
polynomials and Stanley symmetric functions in types A and C,
respectively. $\cW_n$ is the free associative algebra with unity
generated by the elements $u_0,u_1,\ldots,u_{n-1}$ modulo the
relations
\[
\begin{array}{rclr}
u_i^2 & = & 0 & i\geq 0\ ; \\
u_iu_j & = & u_ju_i & |i-j|\geq 2\ ; \\
u_iu_{i+1}u_i & = & u_{i+1}u_iu_{i+1} & i>0\ ; \\
u_0u_1u_0u_1 & = & u_1u_0u_1u_0.
\end{array}
\]
For any $w\in W_n$, choose a reduced word $a_1\cdots a_\ell$ for $w$
and define $u_w = u_{a_1}\ldots u_{a_\ell}$. Since the last three
relations listed are the Coxeter relations for $W_n$, it is clear that
$u_w$ is well defined, and that the $u_w$ for $w\in W_n$ form a free
$\Z$-basis of $\cW_n$. We denote the coefficient of $u_w\in \cW_n$ in
the expansion of the element $f\in \cW_n$ by $\langle f,w\rangle$;
thus $f = \sum_{w\in W_n}\langle f,w\rangle\,u_w$ for all $f\in \cW_n$.

Let $t$ be an indeterminate and define
\begin{gather*}
A_i(t) = (1+t u_{n-1})(1+t u_{n-2})\cdots 
(1+t u_i) \ ; \\
\tilde{A}_i(t) = (1-t u_i)(1-t u_{i+1})\cdots (1-t u_{n-1}) \ ; \\
C(t) = (1+t u_{n-1})\cdots(1+t u_1)
(1+2t u_0)(1+t u_1)\cdots (1+t u_{n-1}).
\end{gather*}
According to \cite[Lemma 2.1]{FS} and \cite[Prop.\ 4.2]{FK2}, for all
commuting variables $s$, $t$ and indices $i$, the relations
$A_i(s)A_i(t) = A_i(t)A_i(s)$ and $C(s)C(t) = C(t)C(s)$ hold. If
$C(X)=C(x_1)C(x_2)\cdots$ and $A(Y)=A_1(y_1)A_1(y_2)\cdots$, we deduce
that the functions $F_w(X)$ and $G_\om(Y)$ defined for $w\in W_n$ and
$\om\in S_n$ by
\[
F_w(X) = \langle C(X), w \rangle \ \ \ \mathrm{and} \ \ \ 
G_\om(Y) = \langle A(Y), \om\rangle 
\]
are symmetric functions in $X$ and $Y$, respectively.  The $G_\om$ and
$F_w$ are the type A and type C Stanley symmetric functions,
introduced in \cite{Sta} and \cite{BH, FK2, L2}. We have that 
$F_w=F_{w^{-1}}$.

When $G_\om$ is expanded in the basis of Schur functions,
one obtains a formula
\begin{equation}
\label{Geq}
G_\om(Y) = \sum_{\la\, :\, |\la| = \ell(\om)}c^\om_\la s_\la(Y)
\end{equation}
for some nonnegative integers $c^\om_\la$ (see \cite{LS2}, \cite{EG}).
According to \cite{LS2} (see also \cite[(7.22)]{M1}), we have
$c_\la^\om = c^{\om^{-1}}_{\la'}$.  Lascoux and Sch\"utzenberger
\cite{LS3} gave one of the first combinatorial interpretations for the
coefficients $c^\om_\la$, as the number of leaves of shape $\la$ in
the {\em transition tree} $T(\om)$ they associated to $\om$. Equation
(\ref{Geq}) may be used to define the double Stanley symmetric
functions $G_\om(Y/Z)$.

For any $\om \in S_n$, the Schubert polynomial $\AS_\om$ of Lascoux
and Sch\"utzenberger is given by 
\begin{equation}
\label{ASdef}
\AS_\om(Y) = \left\langle 
A_1(y_1)A_2(y_2)\cdots A_{n-1}(y_{n-1}), \om\right\rangle.
\end{equation}
The definition (\ref{ASdef}) is equivalent to the
one in \cite{LS1}, as is shown in \cite{FS}.
Now define
\begin{equation}
\label{dbleC}
\CS_w(X\,;Y,Z) = \left\langle 
\tilde{A}_{n-1}(z_{n-1})\cdots \tilde{A}_1(z_1)C(X) A_1(y_1)\cdots 
A_{n-1}(y_{n-1}), w\right\rangle.
\end{equation}
If $\CS_w(X\,;Y):=\CS_w(X\,;Y,0)$,
then (\ref{dbleC}) is equivalent to the equation
\begin{equation}
\label{dbleC2}
\CS_w(X\,;Y,Z) = \sum_{uv=w}\AS_{u^{-1}}(-Z)\CS_v(X\,;Y)
\end{equation}
summed over all reduced factorizations $uv=w$ with $u\in S_n$.  The
polynomials $\CS_w$ in (\ref{dbleC}) were introduced by Ikeda,
Mihalcea, and Naruse \cite{IMN}; they are double versions of the type
C Billey-Haiman Schubert polynomials \cite{BH}. Their definition
differs from (\ref{dbleC}), but the equivalence of the two follows by
combining (\ref{dbleC2}) with \cite[\S 7]{FK2} and \cite[Cor.\
8.10]{IMN}.  One checks that $\AS_\om$ and $\CS_w$ are stable under
the natural inclusion of $W_n$ in $W_{n+1}$, and hence well defined
for $\om\in S_\infty$ and $w\in W_\infty$, respectively. The
$\AS_\om(Y)$ for $\om \in S_\infty$ form a $\Z$-basis of the
polynomial ring $\Z[Y]$, and the $\CS_w(X\,;Y)$ for $w\in W_\infty$
form a $\Z$-basis of $\Gamma[Y]$.

If $\om\in S_\infty$ is a Grassmannian permutation with a unique
descent at $r$, then $\AS_\om(Y)$ is a Schur polynomial in
$(y_1,\ldots,y_r)$. In \cite[\S 6]{BKT2}, we obtained the analogue of
this result for the $\CS_w(X\,;Y)$: if $w = w_\la\in W_\infty$ is the
$k$-Grassmannian permutation associated to the $k$-strict partition
$\la$, then 
\begin{equation}
\label{CtoT}
\CS_{w_\la}(X\,;Y) = \Ti_\la(X\,;Y).
\end{equation}

\subsection{Splitting type A Schubert polynomials}

If $r\leq s$ are any two integers, and $P(X,Y,Z)$ is any polynomial or
formal power series in the variables $x_i$, $y_j$, and $z_j$, we let
$P[r,s]$ denote the power series obtained from $P(X,Y,Z)$ by setting
$x_i=0$ for all $i$ if $0\notin [r,s]$, $y_j=0$ if $j\notin [r,s]$,
and $z_j = 0$ if $-j\notin [r,s]$. If $0\in [r,s]$ we set
$P^{(r,s)}=P[r,s]$ and $P^{(s)}=P[0,s]$.

If $w\in W_n$ and $v\in S_m$, we define $w\times v\in W_{m+n}$ to be
the signed permutation $(w_1,\dots,w_n,v_1+n,\ldots, v_m+n)$; in
particular, $1_n\times v$ is used to denote the permutation
$(1,\dots,n,v_1+n,\ldots, v_m+n)$.  For $\om\in S_\infty$ and $1\leq
r\leq s$, equation (\ref{ASdef}) immediately gives
\begin{equation}
\label{Astab}
\AS_\om[r,s] = \begin{cases} \AS_v(y_r,\ldots,y_s) & \text{if 
$\om=1_{r-1}\times v$}, \\
0 & \text{otherwise}.
\end{cases}
\end{equation}
Given any sequence $\fraka \, :\, a_1<\cdots <a_p$ of positive
integers, we furthermore obtain
\[
\AS_\om(Y) = \sum_{u_1\cdots u_p = \om}\AS_{u_1}[1,a_1]
\AS_{u_2}[a_1+1,a_2]\cdots\AS_{u_p}[a_{p-1}+1,a_p]
\]
summed over all reduced factorizations $u_1\cdots u_p = \om$.

If $\om$ is increasing up to $r$, then $\AS_\om$ is symmetric in
$y_1,\ldots , y_r$ and we have
\begin{equation}
\label{AG}
\AS_\om^{(r)} = G_\om^{(r)} = \sum_{\la\, :\, |\la| = 
\ell(\om)}c^\om_\la s_\la(y_1,\ldots,y_r)
\end{equation} 
with the coefficients $c^\om_\la$ as in (\ref{Geq}). Suppose now that
$\om$ is compatible with the sequence $\fraka$, and set $Y_i =
\{y_{a_{i-1}+1},\ldots,y_{a_i}\}$ for each $i$.  From the previous
considerations, we deduce that $\AS_\om$ satisfies the formula
\begin{equation}
\label{SGsplitting}
\AS_\om(Y) = \sum_{u_1\cdots u_p = \om} G_{u_1}(Y_1)\cdots G_{u_p}(Y_p)
\end{equation}
summed over all reduced factorizations $u_1\cdots u_p = \om$
compatible with $\fraka$. Using (\ref{Geq}) to refine (\ref{SGsplitting})
further gives
\begin{equation}
\label{ASsplitting}
\AS_\om(Y) = \sum_{\underline{\la}} c^\om_{\underline{\la}}
s_{\la^1}(Y_1)\cdots s_{\la^p}(Y_p)
\end{equation}
summed over all sequences of partitions 
$\underline{\la}=(\la^1,\ldots,\la^p)$, where 
\[
c^\om_{\underline{\la}} = \sum_{u_1\cdots u_p = \om}
c_{\la^1}^{u_1}\cdots c_{\la^p}^{u_p}, 
\]
summed over all reduced factorizations $u_1\cdots u_p = \om$
compatible with $\fraka$. More general versions of
(\ref{SGsplitting}), (\ref{ASsplitting}) for universal Schubert
polynomials \cite{Fu4} and quiver polynomials \cite{BF} are
established in \cite{BKTY, KMS}.  Equation (\ref{ASsplitting}) was
used in \cite[\S 5]{BKTY} to obtain Giambelli formulas for type A
partial flag varieties.

\section{Transition for mixed Stanley functions}
\label{transition}

\subsection{Mixed Stanley functions}

\begin{defn}
\label{Stdef}
Given $w\in W_n$, the (right) {\em type C mixed Stanley function}
$J_w(X\,;Y)$ is defined by the equation
\[
J_w(X\,;Y) = 
\langle C(X)A(Y),w\rangle = 
\sum_{uv=w}F_u(X)G_v(Y)
\]
summed over all reduced factorizations $uv=w$ with $v\in S_n$.
\end{defn}

Definition \ref{Stdef} can be easily restated in terms of reduced
decompositions and admissible sequences, along the lines of \cite[Eq.\
(3.2)]{BH}. One has a dual notion of a left mixed Stanley function
$J'_w(X\,;Y) = \sum_{uv=w}G_{u^{-1}}(Y)F_v(X)$, summed over all reduced
factorizations $uv=w$ with $u\in S_n$. This is equivalent to the right
version since clearly $J'_w(X\,;Y) = J_{w^{-1}}(X\,;Y)$. Furthermore,
observe that $J_w(X\,;Y)$ is well defined for $w\in W_\infty$.

\medskip

Recall that $J^{(k)}_w = J_w[0,k]$. 

\begin{lemma}
\label{CSk}
If $w\in W_\infty$ is increasing up to $k$, then $\CS_w^{(k)}=
J_w^{(k)}$. In particular, if $w=w_\la$ is $k$-Grassmannian, then
\begin{equation}
\label{HeqT}
J_{w_\la}^{(k)}(X\,;Y) = \Ti_\la(X\,;Y).
\end{equation}
\end{lemma}
\begin{proof}
Observe that if $w=uv$ is a reduced factorization, then
$\ell(vs_i)=\ell(v)+1$ for all $i<k$, i.e., $v$ is also
increasing up to $k$. It follows that
\[
\CS_w^{(k)} = \sum_{uv=w, \, v\in S_{\infty}} F_u(X)\AS_v^{(k)} =
\sum_{uv=w, \, v\in S_{\infty}} F_u(X)G_v^{(k)} = J_w^{(k)},
\]
as claimed. Equation (\ref{HeqT}) follows from this and (\ref{CtoT}).
\end{proof}

\begin{example}
If $\la$ is a $k$-strict partition, then
\begin{equation}
\label{Jex}
J_{w_\la}(X\,;Y) = \sum_{\mu\subset\la} F_{\la/\mu}(X)s_{\mu'}(Y)
\end{equation}
summed over all $k$-strict partitions $\mu\subset\la$. The function
$F_{\la/\mu}(X)$ in (\ref{Jex}) is the skew $F$-function from
\cite[\S 5]{T4}, which, when non-zero, is equal to $F_{w_\la
w_\mu^{-1}}(X)$.
\end{example}

\subsection{Transition equations}
\label{tes}

For positive integers $i<j$ we define reflections $t_{ij}\in S_\infty$
and $\ov{t}_{ij},\ov{t}_{ii} \in W_\infty$ by their right actions
\begin{align*}
(\ldots,w_i,\ldots,w_j,\ldots)\,t_{ij} &= 
(\ldots,w_j,\ldots,w_i,\ldots), \\
(\ldots,w_i,\ldots,w_j,\ldots)\,\ov{t}_{ij} &= 
(\ldots,\ov{w}_j,\ldots,\ov{w}_i,\ldots), \ \ \mathrm{and} \\
(\ldots,w_i,\ldots)\,\ov{t}_{ii} &= 
(\ldots,\ov{w}_i,\ldots).
\end{align*} 
We let $\ov{t}_{ji} = \ov{t}_{ij}$. According to \cite[Thms.\ 4,
5]{B}, the type C Schubert polynomials $\CS_w=\CS_w(X\,;Y)$
satisfy the recursion formula
\begin{equation}
\label{recurse}
\CS_w = y_r\CS_{wt_{rs}} + \sum_{{1 \leq i < r} \atop {\ell(wt_{rs}t_{ir}) = 
\ell(w)}} \CS_{wt_{rs}t_{ir}} + 
\sum_{{i\geq 1} \atop {\ell(wt_{rs}\ov{t}_{ir}) = 
\ell(w)}} \CS_{wt_{rs}\ov{t}_{ir}}, 
\end{equation}
where $r$ is the last {\em positive} descent of $w$ and $s$ is maximal
such that $w_s < w_r$. If $w$ is increasing up to $k$ and the last
descent $r$ of $w$ satisfies $r>k$, we deduce from (\ref{recurse}),
Lemma \ref{CSk}, and Lemma \ref{Tlemma} below that
\begin{equation}
\label{rec2}
J_w^{(k)} = \sum_{{1 \leq i < r} \atop {\ell(wt_{rs}t_{ir}) = 
\ell(w)}} J_{wt_{rs}t_{ir}}^{(k)} + 
\sum_{{i \geq 1} \atop {\ell(wt_{rs}\ov{t}_{ir}) = 
\ell(w)}} J_{wt_{rs}\ov{t}_{ir}}^{(k)}.
\end{equation}
The definitions of $r$ and $s$ imply that we have
$\ell(wt_{rs})=\ell(w)-1$ in (\ref{recurse}) and (\ref{rec2}). The
next result therefore characterizes the reflections $t=t_{ir}$ and
$t=\ov{t}_{ir}$ such that $\ell(wt_{rs}t)=\ell(w)$.

\begin{lemma}[\cite{Mo, B}]  
\label{abc}
Let $w$ be an element of $W_\infty$. 

\smallskip
\noin
{\em (a)} If $i<j$, then $\ell(wt_{ij})= \ell(w)+1$ if and only if 
$w_i < w_j$, and there is no $p$ with $i<p<j$ and $w_i<w_p<w_j$.

\smallskip
\noin {\em (b)} If $i\leq j$, then $\ell(w\ov{t}_{ij})= \ell(w)+1$ if
and only if {\em (i)} $-w_i<w_j$ and $-w_j<w_i$, {\em (ii)} in case
$i<j$, either $w_i<0$ or $w_j<0$, and {\em (iii)} there is no $p<i$
such that $-w_j<w_p<w_i$, and no $p<j$ such that $-w_i < w_p < w_j$.

\smallskip
\noin {\em (c)} If $i<j$ and $\ell(w\ov{t}_{ij})= \ell(w)+1$, then
each positive descent of $w\ov{t}_{ij}$ is also a positive descent of
$w$.
\end{lemma}

For any $w\in W_\infty$ which is increasing up to $k$, we construct a
rooted tree $T^k(w)$ with nodes given by elements of $W_\infty$ and
root $w$ as follows. Let $r$ be the last descent of $w$. If $w=1$ or
$r=k$, then set $T^k(w)=\{w\}$. Otherwise, let $s = \max(i>r\ |\ w_i <
w_r)$ and $\Phi(w)= \Phi_1(w)\cup \Phi_2(w)$, where
\begin{gather*}
\Phi_1(w)= \{wt_{rs}t_{ir}\ |\ 1\leq i < r \ \ \mathrm{and} \ \ 
\ell(wt_{rs}t_{ir}) = \ell(w) \}, \\
\Phi_2(w)= 
\{wt_{rs}\ov{t}_{ir}\ |\ i\geq 1 \ \ \mathrm{and} \ \
\ell(wt_{rs}\ov{t}_{ir}) = \ell(w) \}.
\end{gather*}
To recursively define $T^k(w)$, we join $w$ by an edge to each $v\in
\Phi(w)$, and attach to each $v\in \Phi(w)$ its tree $T^k(v)$.

We call $T^k(w)$ the {\em $k$-transition tree} of $w$. Clearly 
$T^k(w)$ is a subtree of $T^{k'}(w)$ for any $k'\leq k$. The 
$0$-transition tree of $w$ was studied in \cite{B, BL}. 

\begin{lemma}
\label{Tlemma}
The $k$-transition tree $T^k(w)$ is finite. All the nodes of $T^k(w)$
are increasing up to $k$, and the leaves of $T^k(w)$ are $k$-Grassmannian
elements.
\end{lemma}
\begin{proof}
If $w$ is $k$-Grassmannian then clearly $T^k(w)=\{w\}$. We will show
that if $w$ is increasing up to $k$ and not $k$-Grassmannian, then
$\Phi(w)\neq \emptyset$ and all elements of $\Phi(w)$ are increasing
up to $k$.  Let $r>k$ be the last descent of $w$, $s$ be maximal such
that $w_s < w_r$, and $v=wt_{rs}$.  If $\Phi_1(w)=\emptyset$, then
according Lemma \ref{abc}(a), we must have $w_i > w_s$ for all
$i<r$. If $w_s>0$ then $v\ov{t}_{rr}\in \Phi_2(w)$. If $w_s<0$, let
$i>0$ be minimal such that $w_i>-w_s$.  Then one can easily check
using Lemma \ref{abc}(b) that $v\ov{t}_{ir}\in \Phi_2(w)$. This proves
that $\Phi(w)$ is nonempty.

We next show that all elements of $\Phi(w)$ are increasing up to $k$.
For this, we may assume that $k>0$. If there exists an $i\leq k$ such
that $\ell(vt_{ir}) = \ell(w)$, then Lemma \ref{abc}(a) implies that
$w_i < w_s$ and there is no $j$ with $i < j < r$ and $w_i < w_j <
w_s$. It follows that $vt_{ir}$ is increasing up to $k$.  If
$\ell(v\ov{t}_{ir}) = \ell(w)$ for some $i\leq k$, then by Lemma
\ref{abc}(b)(ii) we must have $w_s <0$, since $w$ is increasing up to
$k$ and hence $v_i=w_i>0$. Moreover, by Lemma \ref{abc}(c), the
descent set of $v\ov{t}_{ir}$ is contained in the descent set of
$v$. We deduce that $v\ov{t}_{ir}$ is also increasing up to $k$.

Finally, it is shown in the proof of \cite[Thm.\ 4]{B} that the
recursion defining $T^0(w)$ terminates after a finite number of steps;
hence $T^k(w)$ is a finite tree. Moreover, we deduce from loc.\ cit.\
that if $w\in W_n$, then all of the nodes of $T^k(w)$ lie in
$W_{n+r}$.
\end{proof}

\begin{defn}
If $w\in W_\infty$ is increasing up to $k$ and $v$ is a leaf of
$T^k(w)$, the {\em shape} of $v$ is the $k$-strict partition $\la$
associated to $v$.  For any $k$-strict partition $\la$, the {\em mixed
Stanley coefficient} $e^w_{\la}$ is equal to the number of leaves of
$T^k(w)$ of shape $\la$.
\end{defn}

The next result is a type C analogue of equation (\ref{AG}).

\begin{thm}
\label{mainthm}
If $w\in W_\infty$ is increasing up to $k$, then we have an expansion
\begin{equation}
\label{CStan}
\CS_w^{(k)} = J_w^{(k)} = \sum_{\la\, :\, |\la| = \ell(w)}
e^w_{\la}\,\Ti_{\la}
\end{equation}
where the sum is over $k$-strict partitions $\la$.
\end{thm}
\begin{proof}
The equality $\CS_w^{(k)} = J_w^{(k)}$ is proved in Lemma \ref{CSk}.
We deduce from (\ref{rec2}) and the definition of $T^k(w)$ that
if $w$ is not $k$-Grassmannian, then 
\[
J_w^{(k)} = \sum_{v\in \Phi(w)} J_{v}^{(k)}.
\]
On the other hand, for any $k$-Grassmannian element $w=w_\la$, we have
$J^{(k)}_{w_\la} = \Ti_\la$, by (\ref{HeqT}).  This completes the
proof of the theorem.
\end{proof}

\begin{example}
\label{treeex}
The $1$-transition tree of $w=3\ov{1}254$ looks as follows.
\[
\includegraphics[scale=0.30]{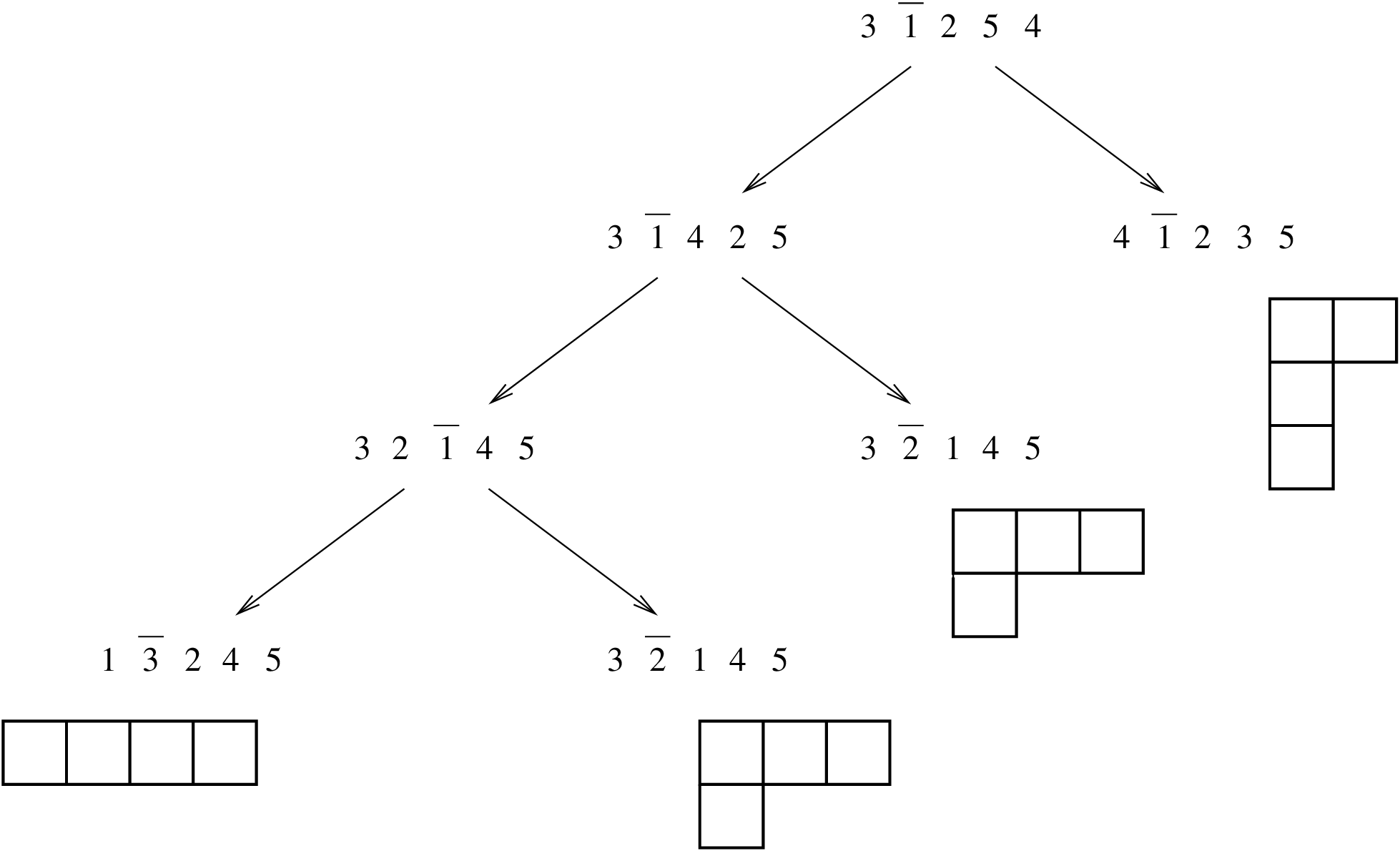}
\]
By Theorem~\ref{mainthm} we therefore obtain
\[
J_{3\ov{1}254}^{(1)} = \Ti_{(2,1,1)} + 2\,\Ti_{(3,1)} + \Ti_4.
\]
\end{example}

\medskip

When $k=0$, Theorem \ref{mainthm} states that for any $w\in W_\infty$,
\begin{equation}
\label{CStan0}
F_w(X) = \sum_{\la\, :\, |\la| = \ell(w)} e^w_{\la}\,Q_{\la}(X)
\end{equation}
summed over strict partitions $\la$. A transition based formula for
the constants $e^w_{\la}$ in equation (\ref{CStan0}) was proved by
Billey \cite{B}. There are several alternative combinatorial
descriptions of these numbers in this case, which include a formula in
terms of {\em Kra\'skiewicz tableaux} \cite{Kr, L2} of shape $\la$.  It
would be interesting to have an analogous tableau formula for the
$e^w_{\la}$ in the general case where $k>0$. Another natural question
is whether the mixed Stanley coefficients can be used to obtain a
Littlewood-Richardson type rule for theta polynomials; some positive
results in this direction are explained in \S \ref{mrs}.

\begin{example}
Consider the {\em double mixed Stanley function}
\[
J_w(X\,;Y/Z) = \sum_{\om uv=w}G_{\om^{-1}}(-Z)F_u(X)G_v(Y)
\]
summed over all reduced factorizations $\om uv=w$ with $\om,v\in
S_\infty$. Fix integers $j,k\geq 0$, and suppose that $w\in W_\infty$
is increasing up to $k$ and $w^{-1}$ is increasing up to $j$. Then
$\CS_w^{(-j,k)}(X\,;Y,Z) = J_w^{(-j,k)}(X\,;Y/Z)$. However, the
analogue of Theorem \ref{mainthm} fails, at least for fixed $k$. For
an example with $j=k=1$, one checks that $\CS_{231}^{(-1,1)}$ cannot
be written as an integer linear combination of
$\CS_{2\ov{1}3}^{(-1,1)}$ and $\CS_{312}^{(-1,1)}$, while $2\ov{1}3$
and $312$ are the only $1$-Grassmannian elements of length two in
$W_\infty$.
\end{example}

\subsection{Reduced words and $k$-bitableaux}
\label{redwds}

The type A Stanley symmetric functions $G_\om$ were used in \cite{Sta}
to express the number of reduced words of a permutation $\om$ in terms
of the numbers $f^\la$ of standard tableaux of shape $\la$, for the
partitions $\la$ which appear in equation (\ref{Geq}). Similarly, the
type C Stanley symmetric functions $F_w$ can be used to compute the
number of reduced words for an element $w\in W_\infty$, as shown in
\cite{H, Kr}. We proceed to give an analogue of these results for the
mixed Stanley functions $J_w$.

Let {\bf P} denote the ordered alphabet
$\{1'<2'<\cdots<k'<1<2<\cdots\}$.  The symbols $1',\ldots,k'$ are
called {\em marked}, while the rest are {\em unmarked}.  Let $\la$ be
a $k$-strict partition.  A {\em $k$-bitableau} $U$ of shape $\la$ is a
filling of the boxes in $\la$ with elements of {\bf P} which is weakly
increasing along each row and down each column, such that the marked
entries are strictly increasing along each row and the unmarked
entries form a {\em $k$-tableau} $T$. We refer to \cite[\S 5]{T4} for
the definition of a $k$-tableau and more details.  Each $k$-bitableau
$U$ has an associated multiplicity $r(U)$, which is a nonnegative
integer.  Let $(xy)^U=\prod_ix_i^{m_i}\prod_jy_j^{n_j}$, where $m_i$
(respectively $n_j$) denotes the number of times that $i$
(respectively $j'$) appears in $U$. According to \cite[Thm.\ 5]{T4},
we have
\begin{equation}
\label{thtab}
\Ti_\la(X\,;Y) = \sum_U 2^{r(U)}(xy)^U
\end{equation}
summed over all $k$-bitableaux $U$ of shape $\la$.

If a $k$-bitableau $U$ contains exactly $m$ marked entries, we say
that $U$ is of {\em type $m$}. $U$ is called {\em standard} if the
entries $1',\ldots,m',1,\ldots,n$ each appear once in $U$ for some
$m$ and $n$; in this case we have $r(U)=n$.

\begin{defn}
Let $\la$ and $\mu$ be $k$-strict partitions with $\mu\subset\la$, and
$m$ be a nonnegative integer with $m\leq k$. We denote by
$g^{\la/\mu}$ the number of standard $k$-tableaux of skew shape
$\la/\mu$, and by $h^\la_m$ the number of standard $k$-bitableaux of
shape $\la$ and type $m$. We say that a reduced word for $w\in
W_\infty$ has {\em type $m$} if the last $m$ letters of the word are
positive.
\end{defn}

\begin{prop}
For any $k$-strict partition $\la$ and integer $m\leq k$, we have
\[
h^\la_m = \sum_{\mu\subset\la,\, |\mu|=m}  f^\mu g^{\la/\mu}
\]
where the sum is over all partitions $\mu\subset\la$ with $|\mu|=m$. 
\end{prop}
\begin{proof}
Suppose that $\la$ is a $k$-strict partition.  Using e.g.\
\cite[Prop.\ 5]{T4}, one can construct a bijection between the set of
standard $k$-bitableaux of shape $\la$ and type $m$ and the set of
reduced words of type $m$ for $w_\la$. According to \cite[Thm.\ 6 and
Ex.\ 9]{T4}, for any $k$-strict partition $\mu$ with $\mu\subset\la$
and $\mu_1\leq k$, the number $g^{\la/\mu}$ of standard $k$-tableaux
of shape $\la/\mu$ is equal to the number of reduced words for $w_\la
w_\mu^{-1}$. Moreover, the number $f^\mu$ of standard tableaux of
shape $\mu$ equals the number of reduced words for the permutation
$w_\mu\in S_\infty$. The result follows.
\end{proof}

The $k=0$ case of the next result is due to Haiman \cite{H} and 
Kra\'skiewicz \cite{Kr}.

\begin{prop}
\label{hkgeneral}
Let $w\in W_\infty$ be increasing up to $k$ and let $m$ be an
integer with $0\leq m \leq \min(k,\ell(w))$. Then the number of
reduced words of type $m$ for $w$ is equal to $\sum_\la e_\la^w
h^\la_m$, where the sum is over all $k$-strict partitions $\la$.
\end{prop}
\begin{proof}
It is clear that the number of reduced words of type $m$ for $w$
equals $2^{-n}$ times the coefficient of $x_1\cdots x_n y_1\cdots y_m$
in $J_w$, where $n=\ell(w)-m$.  On the other hand, this coefficient is
also  equal to $2^n\sum_\la e_\la^w h^\la_m$, by (\ref{CStan}) and
(\ref{thtab}).
\end{proof}

It would be interesting to find a bijective proof of Proposition
\ref{hkgeneral}.

\subsection{Multiplication rules}
\label{mrs}

We show here how Theorem \ref{mainthm} may be used to obtain
Littlewood-Richardson type rules for the product of two theta
polynomials. In the case of $k=0$, we observe that the transition
equations in \cite{B} give a combinatorial rule for the structure
constants in the product of {\em any} two Schur $Q$-functions. This
answers a question of Manivel in the affirmative (compare with
\cite[\S 4]{BL}).

We will actually work with the Schur $P$-functions, which are defined
by the equation $P_\la = 2^{-\ell(\la)} Q_\la$, for any strict
partition $\la$.  Given two strict partitions $\mu$ and $\nu$, there
are nonnegative integers $f^\la_{\mu\nu}$ such that
\[
P_{\mu}P_{\nu} = \sum_{\la}f^{\la}_{\mu\nu}P_{\la}.
\]
The $f_{\mu\nu}^\la$ agree with the Schubert structure constants on
maximal orthogonal Grassmannians $\OG(n,2n+1)$, when $n$ is
sufficiently large. Combinatorial rules for the numbers
$f^{\la}_{\mu\nu}$ may be found in \cite{Sa, W, St1, Sh}.

\begin{prop}
The coefficient $f^{\la}_{\mu\nu}$ is equal to the number of leaves of 
$T^0(w_\la w_\mu^{-1})$ of shape $\nu$, if $\mu\subset \la$, and is equal 
to zero, otherwise.
\end{prop}
\begin{proof}
The structure constants $f^{\la}_{\mu\nu}$ appear in the expansion of
{\em skew Schur $Q$-functions}
\begin{equation}
\label{skeweq}
Q_{\la/\mu} = \sum_{\nu} f^{\la}_{\mu\nu}Q_{\nu}
\end{equation}
in the $Q$-basis (see for example \cite[III.5 and III.8]{M2}).
Following \cite[Thm.\ 8.2]{FK2} and \cite[Cor.\ 6.6]{St2}, the skew
Schur $Q$-functions are known to be equal to certain type C Stanley
symmetric functions. In fact, we have $Q_{\la/\mu}= F_{w_\la
w_\mu^{-1}}$, where $w_\la$ and $w_\mu$ are the $0$-Grassmannian
elements associated to $\la$ and $\mu$ (this is a special case of
\cite[Thm.\ 6]{T4}). The proposition follows from this, using
(\ref{CStan0}) and (\ref{skeweq}).
\end{proof}

Now let $k$ be any nonnegative integer. The following result is an
analogue of \cite[Lemma 16]{BL} for the functions $J_w$.

\begin{lemma}
\label{multlem}
For $w\in W_n$ and $v\in S_\infty$, we have
\[
J_wJ_v = J_{w\times v}.
\]
\end{lemma}
\begin{proof}
Since $a_1\cdots a_\ell$ is a reduced word for $v\in S_m$ if and only
if $(a_1+n)\cdots(a_\ell+n)$ is a reduced word for $1_n\times v$, we
see that $G_{1_n\times v} = G_v$, $F_{1_n\times v} = F_v$, and $J_v =
J_{1_n\times v}$ for each $n\geq 1$.  The reduced words for $w\times
v$ are all obtained by intertwining a reduced word for $w$ with a
reduced word for $1_n\times v$. Moreover, given any reduced
factorization $ab=w\times v$, with $b\in S_\infty$, we have
$a=a_1\times a_2$ and $b=b_1 \times b_2$ where $a_1,b_1\in W_n$ and
$a_2,b_1, b_2\in S_\infty$. We deduce that
\begin{align*}
J_{w\times v}(X\,;Y) &= \sum_{(a_1\times a_2)(b_1\times b_2) = w\times v}
F_{a_1\times a_2}(X) G_{b_1\times b_2}(Y) \\
&= \sum_{(a_1\times a_2)(b_1\times b_2) = w\times v}
F_{a_1}(X)F_{1_n\times a_2}(X) G_{b_1}(Y)G_{1_n\times b_2}(Y) \\
&= \sum_{a_1b_1 = w, \ a_2 b_2 =  v}
F_{a_1}(X)G_{b_1}(Y) F_{a_2}(X) G_{b_2}(Y) \\
&= J_w(X\,;Y) J_v(X\,;Y). \qedhere
\end{align*}
\end{proof}

We obtain a combinatorial rule for multiplying two theta polynomials,
when one of the factors is indexed by a `small' partition (compare
with \cite[Cor.\ 17]{BL}).

\begin{prop}
\label{prodprop}
Let $\mu$ and $\nu$ be $k$-strict partitions with $\nu_i\leq k$ for 
all $i$, and consider the product expansion
\[
\Ti_\mu\Ti_\nu = \sum_\la \varphi_{\mu\nu}^\la \Ti_\la
\]
summed over $k$-strict partitions $\la$. Then $\varphi_{\mu\nu}^\la$
is equal to the number of leaves of $T^k(w_\mu\times w_\nu)$ of shape
$\la$.
\end{prop}
\begin{proof}
Observe that $w_\nu\in S_\infty$ if and only if $\nu_i\leq k$ for all
$i$. Lemma \ref{multlem} therefore applies and gives the equation
\[
\Ti_\mu\Ti_\nu = J^{(k)}_{w_\mu} J^{(k)}_{w_\nu} =
J^{(k)}_{w_{\mu}\times w_{\nu}}.
\]
The result now follows from Theorem \ref{mainthm}. Recall from 
\cite[\S 5.4]{BKT2} that in this situation we have 
$\Ti_\nu = R^0\, \ti_\nu  = \det(\ti_{\nu_i+j-i})_{i,j}$.
\end{proof}

\begin{example} 1) 
For the $1$-strict partitions $\mu=(2,1)$ and $\nu=1$ we have
$w_{\mu}\times w_\nu = 3\ov{1}254\in W_5$. Example \ref{treeex}
therefore gives $\Ti_{(2,1)}\Ti_1 = \Ti_{(2,1,1)} + 2\,\Ti_{(3,1)} +
\Ti_4$.

\medskip
\noindent
2) When $k\geq 1$ and $\nu=(1^p)$ for some $p\geq 0$, Proposition
\ref{prodprop} gives a `Pieri type rule' which evaluates the products
$\Ti_{(1^p)}\Ti_\mu$ for any $k$-strict partition $\mu$. A different
combinatorial rule for the same Pieri products was obtained by Pragacz
and Ratajski \cite{PR1}.
\end{example}

\section{Splitting type C Schubert polynomials}
\label{splitC}

In this section we give splitting theorems for the single and double
type C Schubert polynomials $\CS_w$. For any $k\geq 0$, let
$Y_{>k}=(y_{k+1},y_{k+2},\ldots)$.  The following proposition
generalizes the $k=0$ case from \cite[Thm.\ 3]{BH}.
\begin{prop}
\label{basis}
If $w\in W_\infty$ is increasing up to $k$, then 
\begin{equation}
\label{basic}
\CS_w(X\,;Y) = \sum_{u(1_k\times v) = w}
J^{(k)}_u(X\, ; Y)\AS_v(Y_{>k})
\end{equation}
where the sum is over all reduced factorizations $u(1_k\times v) = w$ with
$v\in S_\infty$. The Schubert polynomials $\CS_w(X\,;Y)$ for $w\in
W_\infty$ increasing up to $k$ form a $\Z$-basis for the ring
$\Gamma^{(k)}[Y_{>k}]=\Gamma^{(k)}[y_{k+1},y_{k+2},\ldots]$.
\end{prop}
\begin{proof}
From the definition (\ref{dbleC}) we deduce that for any $w\in W_n$,
\[
\CS_w(X\,;Y) = \sum_{u\om=w,\, \om\in S_{\infty}}\CS_u[0,k]\AS_\om[k+1,n].
\]
According to (\ref{Astab}), the polynomial $\AS_\om[k+1,n]$ is
non-zero only if $\om=1_k\times v$ for some $v \in S_\infty$, in
which case $\AS_\om[k+1,n] = \AS_v(y_{k+1},\ldots,y_n)$.
For all such $\om$, we furthermore note that the element $u=w\om^{-1}$ is
increasing up to $k$, and hence $\CS_u[0,k]=J_u[0,k]$ by Theorem
\ref{mainthm}. This proves equation (\ref{basic}).

Set $\CS_w=\CS_w(X\,;Y)$ and for each $i\geq 0$, let $\partial_i$
denote the divided difference operator from \cite[\S 2]{BH}. Recall
that $\partial_i\CS_w = \CS_{ws_i}$, if $w_i>w_{i+1}$, and
$\partial_i\CS_w = 0$, otherwise.  Since the $\Ti_\la=\CS_{w_\la}$ for
$k$-strict partitions $\la$ form a $\Z$-basis of $\Gamma^{(k)}$, we
deduce that $\partial_i f=0$ for all $i<k$ and $f$ in the ring
$\Gamma^{(k)}[Y_{>k}]$.

Equations (\ref{CStan}) and (\ref{basic}) imply that the $\CS_w$ for $w$
increasing up to $k$ are contained in $\Gamma^{(k)}[Y_{>k}]$. Moreover, the
$\CS_w$ for $w\in W_\infty$ are known to be a $\Z$-basis of
$\Gamma[Y]$ from \cite[Thm.\ 3]{BH}. Given any $f\in
\Gamma^{(k)}[Y_{>k}]$, we therefore have $f = \sum_{w\in
W_\infty}a_w\,\CS_w$ for some $a_w\in \Z$. Since $\partial_if=0$ for
all $i<k$, we deduce that only terms $\CS_w$ with $w$ increasing up to $k$
appear in the sum, completing the proof.
\end{proof}

Fix a sequence $\fraka \, :\, a_1 < \cdots < a_p$ of nonnegative
integers and $w \in W_\infty$ which is compatible with $\fraka$.
Given any sequence of partitions
$\underline{\la}=(\la^1,\ldots,\la^p)$ with $\la^1$ $a_1$-strict, we
define the nonnegative integer
\begin{equation}
\label{fdef0}
e^w_{\underline{\la}} = \sum_{u_1\cdots u_p = w}
e_{\la^1}^{u_1}c_{\la^2}^{u_2}\cdots c_{\la^p}^{u_p}, 
\end{equation}
where the sum is over reduced factorizations $u_1\cdots u_p = w$
compatible with $\fraka$ such that $u_2,\ldots,u_p\in S_\infty$, and
the integers $c_{\la^i}^{u_i}$ and $e^{u_1}_{\la^1}$ are as in
(\ref{Geq}) and (\ref{CStan}), respectively.

The number $e^w_{\underline{\la}}$ can be described in a more
picturesque way as follows. Given a permutation $\om\in S_\infty$, let
$T(\om)$ denote the Lascoux-Sch\"utzenberger transition tree
associated to $\om$ in \cite[\S 4]{LS3}. For any reduced factorization
$u_1\cdots u_p = w$ compatible with $\fraka$ such that
$u_2,\ldots,u_p\in S_\infty$, the $p$-tuple of trees
$(T^{a_1}(u_1),T(u_2),\ldots,T(u_p))$ is called a {\em grove}. The
collection of all such groves forms the {\em $\fraka$-transition
forest} associated to $w$. The integer $e^w_{\underline{\la}}$ is then
equal to the number of $p$-tuples of leaves of shape $\underline{\la}$
in the groves of the $\fraka$-transition forest associated to $w$.

Define $Y_i = \{y_{a_{i-1}+1},\ldots,y_{a_i}\}$ for each
$i\geq 1$; in particular $Y_1=\emptyset$ if $a_1=0$. 

\begin{thm}
\label{CSHG}
Suppose that $w\in W_\infty$ is compatible with the sequence
$\fraka$. Then we have
\begin{equation}
\label{CSHGsplitting}
\CS_w(X\,;Y) = \sum_{u_1\cdots u_p = w} J_{u_1}(X\,;Y_1)G_{u_2}(Y_2)
\cdots G_{u_p}(Y_p)
\end{equation}
summed over all reduced factorizations $u_1\cdots u_p = w$ compatible
with $\fraka$ such that $u_2,\ldots,u_p\in S_\infty$. Furthermore, 
we have
\begin{equation}
\label{CSsplitting}
\CS_w(X\,;Y) = \sum_{\underline{\la}} 
e^w_{\underline{\la}}\,
\Ti_{\la^1}(X\,;Y_1)s_{\la^2}(Y_2)\cdots s_{\la^p}(Y_p)
\end{equation}
summed over all sequences of partitions $\underline{\la}=(\la^1,\ldots,\la^p)$
with $\la^1$ $a_1$-strict.
\end{thm}
\begin{proof}
Equation (\ref{CSHGsplitting}) follows from (\ref{basic}) and the type
A Schubert splitting formula (\ref{SGsplitting}). Moreover,
(\ref{CSsplitting}) is obtained from (\ref{CSHGsplitting}) by using
equations (\ref{Astab}), (\ref{AG}), and (\ref{CStan}).
\end{proof}

The $k=0$ case of formula (\ref{CSsplitting}) is contained in \cite[\S
5]{Yo}. It is expressed in loc.\ cit.\ using a combinatorial
interpretation of the coefficients $c^\om_\la$ in (\ref{Geq}) due to
Fomin and Greene \cite{FG}, which was also exploited in \cite[Cor.\
3]{BKTY}.

Fix a second sequence $\frakb \, :\, 0=b_1 < \cdots <b_q$ of
nonnegative integers such that $w^{-1}$ is compatible with $\frakb$. A
reduced factorization $u_1\cdots u_{p+q-1} = w$ is {\em compatible
with} $\fraka$, $\frakb$ if $u_j(i)=i$ whenever $j<q$ and $i\leq
b_{q-j}$ or whenever $j>q$ and $i \leq a_{j-q}$.  Given any sequence
of partitions $\underline{\la}=(\la^1,\ldots,\la^{p+q-1})$ with
$\la^q$ $a_1$-strict, we define
\begin{equation}
\label{dbfdef0}
f^w_{\underline{\la}} = \sum_{u_1\cdots u_{p+q-1} = w}
c_{\la^1}^{u_1}\cdots c_{\la^{q-1}}^{u_{q-1}}
e_{\la^q}^{u_q}c_{\la^{q+1}}^{u_{q+1}}\cdots c_{\la^{p+q-1}}^{u_{p+q-1}}, 
\end{equation}
where the sum is over reduced factorizations $u_1\cdots u_{p+q-1} = w$
compatible with $\fraka$, $\frakb$ such that $u_i\in S_\infty$ for all
$i\neq q$. Set $Z_j= \{z_{b_{j-1}+1},\ldots,z_{b_j}\}$ for each
$j$.

\begin{cor}
\label{dbleCSHG}
Suppose that $w$ and $w^{-1}$ are compatible with
the sequences $\fraka$ and $\frakb$, respectively. Then 
$\CS_w(X\,;Y,Z)$ is equal to 
\[
\sum_{u_1\cdots u_{p+q-1} = w} G_{u_1}(0/Z_q)\cdots
G_{u_{q-1}}(0/Z_2) J_{u_q}(X\,;Y_1) G_{u_{q+1}}(Y_2) \cdots
G_{u_{p+q-1}}(Y_p)
\]
summed over all reduced factorizations $u_1\cdots u_{p+q-1} = w$
compatible with $\fraka$, $\frakb$ such that $u_i\in S_\infty$ for
all $i\neq q$. Furthermore, we have
\begin{equation}
\label{dbCSsplitting}
\CS_w(X\,;Y,Z) = \sum_{\underline{\la}} 
f^w_{\underline{\la}}\,
s_{\la^1}(0/Z_q)\cdots
\Ti_{\la^q}(X\,;Y_1)\cdots s_{\la^{p+q-1}}(Y_p)
\end{equation}
summed over all sequences of partitions 
$\underline{\la}=(\la^1,\ldots,\la^{p+q-1})$
with $\la^q$ $a_1$-strict.
\end{cor}
\begin{proof}
The result follows immediately from equations (\ref{dbleC2}),
(\ref{SGsplitting}), and Theorem \ref{CSHG}.
\end{proof}

According to \cite[Thm.\ 8.1]{IMN}, the Schubert polynomials
$\CS_w(X\,;Y,Z)$ enjoy the following symmetry property:
\begin{equation}
\label{sym}
\CS_w(X\,; Y,Z) = \CS_{w^{-1}}(X\,; -Z,-Y).
\end{equation}
This also follows immediately from equation (\ref{dbleC}). By applying
Corollary \ref{dbleCSHG} to the right hand side of (\ref{sym}), we
obtain dual versions of these splitting equations.

\section{Symplectic degeneracy loci}
\label{sdl}

\subsection{Isotropic partial flag bundles}
Let $E\to B$ be a vector bundle of rank $2n$ on an algebraic variety
$B$. Assume that $E$ is a {\em symplectic} bundle, i.e.\ $E$ is
equipped with an everywhere nondegenerate skew-symmetric form
$E\otimes E\to \C$. A subbundle $V$ of $E$ is {\em isotropic} if the
form vanishes when restricted to $V$; the ranks of isotropic
subbundles of $E$ range from $0$ to $n$. Fix a sequence $\fraka \, :\,
a_1 < \cdots < a_p$ of nonnegative integers with $a_p <n$, and set
$a_{p+1}=n$ for convenience. We introduce the {\em isotropic partial
flag bundle} $\F^\fraka(E)$ with its projection map $\rho:\F^\fraka(E)
\to B$. The variety $\F^\fraka(E)$ parametrizes partial flags
\begin{equation}
\label{Edot}
E_\bull \ :\ 0= E_0 \subset E_1 \subset \cdots \subset E_p \subset E
\end{equation}
with $\rank E_i = n-a_{p+1-i}$ and $E_p$ isotropic. Here we have
identified $E$ with its pullback under the map $\rho$, and also 
let (\ref{Edot}) denote the tautological partial flag of vector bundles
over $\F^\fraka(E)$.

There is a group monomorphism $\phi:W_n\hra S_{2n}$ with image
\[
\phi(W_n)=\{\,\om\in S_{2n} \ |\ \om(i)+\om(2n+1-i) = 2n+1,
 \ \ \text{for all}  \ i\,\}.
\]
The map $\phi$ is determined by setting,
for each $w=(w_1,\ldots,w_n)\in W_n$ and $1\leq i \leq n$,
\[
\phi(w)(i)=\left\{ \begin{array}{cl}
             n+1-w_{n+1-i} & \mathrm{ if } \ w_{n+1-i} \ \mathrm{is} \
             \mathrm{unbarred}, \\
             n+\ov{w}_{n+1-i} & \mathrm{otherwise}.
             \end{array} \right.
\]
Let $W^\fraka$ be the set of signed permutations $w\in W_n$ whose
descent positions are listed among the integers $a_1,\ldots,
a_p$. These elements are the minimal length coset representatives in
$W_n/W_\fraka$, where $W_\fraka$ denotes the subgroup of $W_n$
generated by the simple reflections $s_i$ for $i\notin \{a_1,\ldots,
a_p\}$. 

Fix a flag $0=F_0\subset F_1\subset \cdots \subset F_n \subset E$ of
subbundles of $E$ with $\rank F_i = i$ for each $i$ and $F_n$
isotropic. We extend any such flag to a complete flag $F_{\bull}$ in
$E$ by letting $F_{n+i}=F_{n-i}^{\perp}$ for $1\leq i\leq n$; we call
the completed flag a {\em complete isotropic flag}.  For every $w\in
W^{\fraka}$ and complete isotropic flag $F_\bull$ of subbundles of
$E$, we define the {\em universal Schubert variety} $\X_w\subset
\F^\fraka(E)$ as the locus of $a \in \F^\fraka(E)$ such that
\begin{equation}
\label{schdef}
\dim(E_r(a)\cap F_s(a))\geq \#\,\{\,i \leq \rank 
E_r \ |\ \phi(w)(i)> 2n-s\,\} \ \, \forall \ r,s.
\end{equation}
The variety $\X_w$ is an irreducible subvariety of $\F^\fraka(E)$ of
codimension $\ell(w)$, and may be regarded as a universal degeneracy
locus. Formulas for the classes $[\X_w]$ in the cohomology or Chow
ring of $\F^\fraka(E)$ pull back to identities for corresponding loci
whenever one has a symplectic vector bundle $V$ and two flags of
isotropic subbundles of $V$, following \cite{Fu3}. Moreover, they are
equivalent to formulas which represent the Schubert classes in the
$T$-equivariant cohomology ring of isotropic partial flag varieties,
as observed e.g.\ in \cite{Gra}.

\subsection{Full flag bundles and the geometrization map}

Consider the full flag bundle $\F(E)=\F^{(0,1,2,\ldots,n-1)}(E)$
parametrizing complete isotropic flags of subbundles $E_\bull$ in $E$.
Let $\XX = (\x_1,\ldots,\x_n)$ and $\YY = (\y_1,\ldots,\y_n)$.
According to \cite[\S 3]{Fu2}, the cohomology (or Chow) ring
$\HH^*(\F(E),\Z)$ is presented as a quotient
\begin{equation}
\label{presentation}
\HH^*(\F(E)) \cong \HH^*(B)[\XX,\YY]/\J_n,
\end{equation}
where $\J_n$ denotes the ideal generated by the differences
$e_i(\x_1^2,\ldots,\x_n^2)- e_i(\y_1^2,\ldots,\y_n^2)$ for $1\leq i
\leq n$. The inverse of the isomorphism (\ref{presentation}) sends the
class of $\x_i$ to $-c_1(E_{n+1-i}/E_{n-i})$ and of $\y_i$ to
$-c_1(F_{n+1-i}/F_{n-i})$ for each $i$ with $1\leq i \leq n$.

The ring $\HH^*(\F(E))$ may be used to study the cohomology of any
isotropic partial flag bundle $\F^\fraka(E)$, because the natural
surjection $\F(E)\to \F^\fraka(E)$ induces an injective ring
homomorphism $\iota :\HH^*(\F^\fraka(E))\to\HH^*(\F(E))$. The
tautological vector bundles $E_i$, $F_j$, the universal Schubert
varieties, and their cohomology classes on $\F^\fraka(E)$ pull back
under $\iota$ to the homonymous objects over $\F(E)$.

The Schubert varieties $\X_w$ on $\F(E)$ are indexed by $w$ in the
full Weyl group $W_n$.  Furthermore, the type C double Schubert
polynomials $\CS_w(X;Y,Z)$ represent their cohomology classes $[\X_w]$
in the presentation (\ref{presentation}), but only after a certain
change of variables. Ikeda, Mihalcea, and Naruse \cite[\S 10]{IMN}
provide a different way to connect these Schubert polynomials to
geometry, which we will adapt to our current setup. For a closely
related approach (which preceded \cite{IMN}) in the case of single
Schubert polynomials, see \cite{T2}.

The key tool is the following ring homomorphism derived from
\cite{IMN}, which we call the {\em geometrization map}\,:
\[
\pi_n : \Gamma[Y,Z] \to \HH^*(B)[\XX,\YY]/\J_n.
\]
The homomorphism $\pi_n$ is determined by setting
\begin{gather*}
\pi_n(q_r(X))=\sum_{i=0}^r e_i(\XX)h_{r-i}(\YY) \ \ \text{for all} \
r, \\ \pi_n(y_i)=\begin{cases} -\x_i & \text{if $1\leq i\leq n$}, \\
\ \ 0 & \text{if $i>n$}, \end{cases}
\ \ \ \text{and} \ \ \ \pi_n(z_j)= \begin{cases}
\y_j & \text{if $1\leq j\leq n$}, \\
0 & \text{if $j>n$}. \end{cases}
\end{gather*}
It follows from \cite[\S 1]{Gra} and \cite[Prop.\ 7.7 and \S 10]{IMN}
that for $w\in W_n$, the geometrization map $\pi_n$ maps
$\CS_w(X\,;Y,Z)$ to a polynomial which represents the universal
Schubert class $[\X_w]$ in the presentation
(\ref{presentation}). Furthermore, we have $\pi_n(\CS_w)\in \J_n$ for
$w\in W_\infty\ssm W_n$.

If $V$ and $V'$ are two vector bundles with total Chern classes $c(V)$
and $c(V')$, respectively, we denote $s_\la(c(V)-c(V'))$ by
$s_\la(V-V')$. We similarly denote the class $\Ti_\la(c(V)-c(V'))$ by
$\Ti_\la(V-V')$.  To state our main geometric result, let $Q_1 =
E/E_p, Q_2= E_p/E_{p-1},\ldots, Q_p =E_2/E_1$. Consider a sequence
$\frakb \, :\, 0=b_1 < \cdots <b_q$ with $b_q <n$, and set
$\wh{Q}_2 = F_n/F_{n-b_2}, \ldots, \wh{Q}_q=
F_{n-b_{q-1}}/F_{n-b_q}$.

\begin{thm}
\label{dbleCloci}
Suppose that $w\in W^{\fraka}$ and that $w^{-1}$ is compatible with 
$\frakb$. Then we have
\begin{align*}
[\X_w] &= \sum_{\underline{\la}} f^w_{\underline{\la}}\,
s_{(\la^1)'}(\wh{Q}_q)\cdots s_{(\la^{q-1})'}(\wh{Q}_2)
\Ti_{\la^q}(Q_1-F_n) s_{\la^{q+1}}(Q_2)\cdots s_{\la^{p+q-1}}(Q_p) \\
&= \sum_{\underline{\la}} f^w_{\underline{\la}}\,
s_{\la^1}(F_{n+b_{q-1}} - F_{n+b_q})\cdots
\Ti_{\la^q}(E-E_p-F_n)
\cdots s_{\la^{p+q-1}}(E_2-E_1)
\end{align*}
in $\HH^*(\F^\fraka(E))$, where the sum is over all sequences of
partitions $\underline{\la}=(\la^1,\ldots,\la^{p+q-1})$ with $\la^q$
$a_1$-strict, and the coefficients $f^w_{\underline{\la}}$ are given
by {\em (\ref{dbfdef0})}.
\end{thm}
\begin{proof}
The variables $\x_i$ for $1\leq i \leq n$ give the Chern roots of the
various vector bundles over $\F^\fraka(E)$. In particular the Chern
roots of $Q_1$ are $\x_1,\ldots,\x_n,-\x_1,\ldots,-\x_{a_1}$, while
those of $Q_r$ for $r\geq 2$ are
$-\x_{a_{r-1}+1},\ldots,-\x_{a_r}$. Similarly the Chern roots of
$F_{n+1-r}$ are represented by $-\y_r,\ldots,-\y_n$ for each $r$. With
$k=a_1$ we have $\ti_r(X\,;Y_1) = \sum_{i=0}^r
q_{r-i}(X)e_i(y_1,\ldots,y_{a_1})$ for any $r\geq 0$. Therefore, we
obtain
\begin{align*}
\pi_n(\ti_r(X\,;Y_1)) &= \sum_{i,j\geq 0} e_{r-i-j}(\XX)
h_j(\YY) e_i(-\x_1,\ldots,-\x_{a_1}) \\ 
&= \sum_{j\geq 0} e_{r-j}(\x_1,\ldots,\x_n, -\x_1,\ldots,-\x_{a_1}) 
h_j(\YY)= c_r(Q_1-F_n)
\end{align*}
and hence 
\[
\pi_n(\Ti_\la(X\,;Y_1)) = \Ti_\la(Q_1-F_n)
\]
for any $a_1$-strict partition $\la$. Moreover, for any partition
$\mu$ and $r\geq 2$, we have
\[
\pi_n(s_\mu(Y_r)) = s_\mu(-\x_{a_{r-1}+1},\ldots,-\x_{a_r}) = s_\mu(Q_r),
\] 
while
\[
\pi_n(s_\mu(0/Z_r)) = s_{\mu'}(-\y_{b_{r-1}+1},\ldots,-\y_{b_r}) =
s_{\mu'}(\wh{Q}_r) = s_\mu(F_{n+b_{r-1}}-F_{n+b_r}).
\]
We deduce that $\pi_n$ maps formula
(\ref{dbCSsplitting}) onto the desired equality.
\end{proof}

\section{Giambelli formulas for symplectic flag varieties}
\label{flags}

\subsection{Partial isotropic flag varieties}
Equip the vector space $E=\C^{2n}$ with a nondegenerate skew-symmetric
bilinear form. Fix a sequence $\fraka \, :\, a_1 < \cdots < a_p$ of
nonnegative integers with $a_p <n$, and set $a_{p+1}=n$. Let
$\X(\fraka)$ be the variety parametrizing partial flags of subspaces
\[
0= E_0 \subset E_1 \subset \cdots \subset E_p \subset E
\]
with $\dim E_i = n-a_{p+1-i}$ and $E_p$ isotropic.
The same notation will be used to denote the tautological partial 
flag $E_\bull$ of vector bundles over $\X(\fraka)$.

A presentation of the cohomology ring of $\X(\fraka)$ as a quotient of
the symmetric algebra on the characters of a maximal torus in
$\Sp_{2n}$ is well known \cite{Bo}. We will give here an alternative
presentation using the Chern classes of the tautological vector
bundles over $\X(\fraka)$. Let $Q_1 = E/E_p, Q_2= E_p/E_{p-1},\ldots,
Q_{p+1} =E_1$ and set $\sigma_i = c_i(Q_1)$ for $1\leq i \leq n+a_1$
and $c_j^r =c_j(Q_r)$ for $2\leq r\leq p+1$ and $1\leq j\leq
a_r-a_{r-1}$.  We then have the following result

\begin{prop}
\label{pres}
The cohomology ring $\HH^*(\X(\fraka),\Z)$ is presented as a quotient of the 
polynomial ring $\Z[\s_1,\ldots,\s_{n+a_1},c_1^2,\ldots,c_{a_2-a_1}^2,
\ldots,c_1^{p+1},\ldots,c_{n-a_p}^{p+1}]$ modulo the relations
\[
\det(\sigma_{1+j-i})_{1\leq i,j \leq r} =
(-1)^r\sum_{i_2+\cdots+i_{p+1}=r} c_{i_2}^2\cdots c_{i_{p+1}}^{p+1}\,,
\ \ \ \ 1\leq r \leq n+a_1
\]
and
\begin{equation}
\label{R2}
\sigma_r^2 + 2\sum_{i=1}^{n+a_1-r}(-1)^i \sigma_{r+i}\sigma_{r-i}= 0\,,
 \ \  \ \ a_1+1\leq r \leq n.
\end{equation}
\end{prop}
\begin{proof}
Let $\IG=\IG(n-a_1,2n)$ be the Grassmannian parametrizing isotropic
subspaces of $E$ of dimension $n-a_1$. According to \cite[Thm.\
1.2]{BKT1}, the cohomology ring of $\IG$ is isomorphic to the
polynomial ring generated by the Chern classes $\sigma_i=c_i(Q_1)$ and
the Chern classes of $E_p$, modulo the relations
\[
\det(\sigma_{1+j-i})_{1\leq i,j \leq r} = (-1)^r c_r(E_p)\,, \ \ \ \
1\leq r \leq n+a_1
\]
coming from the Whitney sum formula $c(E_p)c(Q_1)=1$, as well as the
relations (\ref{R2}).  The map $\X(\fraka)\to \IG$ sending $E_\bull$
to $E_p$ realizes $\X(\fraka)$ as a fiber bundle over $\IG$ with fiber
a partial $\SL_{n-a_1}$ flag variety. We deduce using e.g.\
\cite[Thm.\ 1]{Gr} that $\HH^*(\X(\fraka))$ is isomorphic to the
polynomial ring
$\HH^*(\IG)[c_1^2,\ldots,c_{a_2-a_1}^2,\ldots,c_1^{p+1},
\ldots,c_{n-a_p}^{p+1}]$ modulo the relations
\[
\sum_{i_2+\cdots+i_{p+1}=r} c_{i_2}^2\cdots c_{i_{p+1}}^{p+1} =
c_r(E_p) \,, \ \ \ \ 1\leq r \leq n-a_1.
\] 
The proposition follows by combining these two facts.
\end{proof}

\subsection{Giambelli formulas}

Our choice of the special Schubert classes on symplectic and
orthogonal Grassmannians agrees with the conventions in \cite{BKT1}.
Let $F_\bull$ be a fixed complete isotropic flag of subspaces in
$\C^{2n}$.  For each $w\in W^\fraka$, the Schubert variety
$\X_w(F_\bull)$ in $\X(\fraka)$ is defined by the same equation
(\ref{schdef}) as before. Let $E'_j = E/E_{p+1-j}$ for $1\leq j\leq
p$. The Chern classes $c_i(E'_j)$ for all $i, j$ are the Schubert
classes on $\X(\fraka)$ which are pullbacks of special Schubert
classes on symplectic Grassmannians. By definition, they are the {\em
special Schubert classes} on $\X(\fraka)$, and they generate the
cohomology ring $\HH^*(\X(\fraka))$ by Proposition \ref{pres}.
Specializing Theorem \ref{dbleCloci} to the case when the base $B$ is
a point gives the next result.

\begin{cor}
For every $w\in W^{\fraka}$, we have 
\begin{align*}
\label{giambelli1}
[\X_w] & = \sum_{\underline{\la}} 
e^w_{\underline{\la}}\,
\Ti_{\la^1}(Q_1)s_{\la^2}(Q_2)\cdots s_{\la^p}(Q_p) \\
&= \sum_{\underline{\la}} 
e^w_{\underline{\la}}\,
\Ti_{\la^1}(E'_1)s_{\la^2}(E'_2-E'_1)\cdots s_{\la^p}(E'_p-E'_{p-1})
\end{align*}
in $\HH^*(\X(\fraka),\Z)$, where the sums are over all 
sequences of
partitions $\underline{\la}=(\la^1,\ldots,\la^p)$ with $\la^1$
$a_1$-strict, and the coefficients $e^w_{\underline{\la}}$ are given
by {\em (\ref{fdef0})}.
\end{cor}

\section{Splitting orthogonal Schubert polynomials}
\label{ogps}

In this final section we discuss the form of our splitting results for
the orthogonal groups; we will work throughout with coefficients in
the ring $\Z[\frac{1}{2}]$. For $w\in W_\infty$, let $s(w)$ denote the
number of $i$ such that $w(i)<0$. It follows e.g.\ from \cite{BH, IMN}
that the polynomials $\BS_w=2^{-s(w)}\CS_w$ represent the Schubert
classes in the (equivariant) cohomology ring of odd orthogonal flag
varieties.  Therefore the solutions to the Schubert polynomial
splitting and Giambelli problems for types B and C are essentially the
same. We will describe the splitting theorems for the even orthogonal
groups below; the story is entirely analogous to the symplectic case.

The elements of the Weyl group $\wt{W}_n$ for the root system
$\text{D}_n$ may be represented by signed permutations, as in e.g.\
\cite{B, KT}. The group $\wt{W}_n$ is an extension of $S_n$ by an
element $s_0$ which acts on the right by
\[
(u_1,u_2,\ldots,u_n)s_0=(\ov{u}_2,\ov{u}_1,u_3,\ldots,u_n).
\]
Let $\wt{W}_\infty = \cup_n \wt{W}_n$ and
$\N=\{0, 1, 2, \ldots\}$. A {\em reduced word} of
$w\in\wt{W}_\infty$ is a sequence $a_1\cdots a_\ell$ of elements in
$\N$ such that $w=s_{a_1}\cdots s_{a_\ell}$ and $\ell=\ell(w)$.  We
say that $w$ has a {\em descent} at position $r\geq 0$ if
$\ell(ws_r)<\ell(w)$, where $s_r$ is the simple reflection indexed by
$r$.  
If $k\geq 2$, we say that an element $w\in \wt{W}_\infty$ is {\em
increasing up to $k$} if it has no descents less than $k$; this means
that $|w_1|< w_2<\cdots < w_k$. We also agree that every element of
$\wt{W}_\infty$ is both increasing up to $0$ and increasing up to $1$.

For $k\in \N\ssm\{1\}$, an element $w\in\wt{W}_\infty$ is
$k$-Grassmannian if $\ell(ws_i)=\ell(w)+1$ for all $i\neq k$. We say
that $w$ is $1$-Grassmannian if $\ell(ws_i)=\ell(w)+1$ for all $i\geq
2$. A {\em typed $k$-strict partition} is a pair consisting of a
$k$-strict partition $\la$ together with an integer in $\{0,1,2\}$
called the {\em type} of $\la$, and denoted $\type(\la)$, such that
$\type(\la)>0$ if and only if $\la_i=k$ for some $i\geq 1$. The
geometric significance of the type of $\la$ is explained in 
\cite[\S 4.5]{BKT1}.

Given a $k$-Grassmannian element $w\in \wt{W}_n$, there exist unique
strict partitions $u,\zeta,v$ of lengths $k$, $r$, and $n-k-r$,
respectively, so that
\[
w=(\wh{u}_k,\ldots,u_1,
\ov{\zeta}_1,\ldots,\ov{\zeta}_r,v_{n-k-r},\ldots,v_1)
\]
where $\wh{u}_k$ is equal to $u_k$ or $\ov{u}_k$, according to the
parity of $r$. If 
\[
\mu_i=u_i+i-k-1+\#\{j\ |\ \zeta_j > u_i\},
\]
then $w$ corresponds to a typed $k$-strict partition $\la$ such that
the lengths of the first $k$ columns of $\la$ are given by
$\mu_1,\ldots,\mu_k$. The part of $\la$ in columns $k+1$ and higher is
given by $(\zeta_1-1,\ldots,\zeta_r-1)$; here it is possible that
$\zeta_r=1$, so that the sequence ends with a zero. Finally, if
$\type(\la)>0$, then $\wh{u}_k$ is unbarred if and only if
$\type(\la)=1$.  This defines a bijection between the $k$-Grassmannian
elements of $\wt{W}_\infty$ and the set of all typed $k$-strict
partitions. We let $w_\la$ denote the element of $\wt{W}_\infty$
associated to the typed $k$-strict partition $\la$.

Following \cite{L1}, we will use the nilCoxeter algebra $\wt{\cW}_n$
of $\wt{W}_n$ to define type D Stanley symmetric functions.
$\wt{\cW}_n$ is the free associative algebra with unity generated by
the elements $u_0,u_1,\ldots,u_{n-1}$ modulo the relations
\[
\begin{array}{rclr}
u_i^2 & = & 0 & i\geq 0\ ; \\
u_0 u_1 & = & u_1 u_0 \\
u_0 u_2 u_0 & = & u_2 u_0 u_2 \\
u_iu_{i+1}u_i & = & u_{i+1}u_iu_{i+1} & i>0\ ; \\
u_iu_j & = & u_ju_i & j> i+1, \ \text{and} \ (i,j) \neq (0,2).
\end{array}
\]
For any $w\in \wt{W}_n$, choose a reduced word $a_1\cdots a_\ell$ for $w$
and define $u_w = u_{a_1}\ldots u_{a_\ell}$.  We denote the
coefficient of $u_w\in \wt{\cW}_n$ in the expansion of the element
$f\in \wt{\cW}_n$ by $\langle f,w\rangle$.  Let $t$ be an
indeterminate and, following \cite[4.4]{L1}, define
\begin{gather*}
D(t) = (1+t u_{n-1})\cdots (1+t u_2)(1+t u_1)(1+t u_0)
(1+t u_2)\cdots (1+t u_{n-1}).
\end{gather*}
According to \cite[Lemma 4.24]{L1}, we have $D(s)D(t) = D(t)D(s)$ for 
any commuting variables $s$, $t$. If $D(X)=D(x_1)D(x_2)\cdots$,
then the functions $E_w(X)$ defined by
\[
E_w(X) = \langle D(X), w \rangle 
\]
are the type D Stanley symmetric functions, in agreement with \cite[\S
3]{BH}.

Next, define
\[
\DS_w(X\,;Y,Z) = \left\langle 
\tilde{A}_{n-1}(z_{n-1})\cdots \tilde{A}_1(z_1) D(X) A_1(y_1)\cdots 
A_{n-1}(y_{n-1}), w\right\rangle.
\]
The polynomials $\DS_w(X\,;Y):=\DS_w(X\,;Y,0)$ are the type D
Billey-Haiman Schubert polynomials, and the $\DS_w(X\,;Y,Z)$ are their
double versions studied in \cite{IMN}.
If $w=w_\la$ is $k$-Grassmannian, then $\DS_{w_\la}(X\,;Y)$ is equal
to an {\em eta polynomial} $H_\la(X\,;Y)$; the $H_\la$ are defined
using raising operator expansions analogous to (\ref{Tidef}) in
\cite{BKT4}. When $k=0$, we have that $\la$ is a strict partition and
$H_\la(X\,;Y)=E_{w_\la}(X)=P_\la(X)$ is a Schur $P$-function.

\begin{defn}
Given $w\in \wt{W}_n$, the {\em type D mixed Stanley function}
$I_w(X\,;Y)$ is defined by the equation
\[
I_w(X\,;Y) = 
\langle D(X)A(Y),w\rangle = 
\sum_{uv=w}E_u(X)G_v(Y)
\]
summed over all reduced factorizations $uv=w$ with $v\in S_n$.
\end{defn}

One checks that if $w$ is increasing up to $k$, then
$I^{(k)}_w=\DS_w^{(k)}(X\,;Y)$; in particular, if $w=w_\la$ is
$k$-Grassmannian, then $I^{(k)}_{w_\la} = H_\la$. Furthermore, it
follows from \cite[Thms.\ 4, 5]{B} and the type D analogue of Lemma
\ref{Tlemma} that the $I_w$ for $w$ increasing up to $k$ satisfy the
transition equations
\[
I_w^{(k)} = \sum_{{1 \leq i < r} \atop {\ell(wt_{rs}t_{ir}) = 
\ell(w)}} I_{wt_{rs}t_{ir}}^{(k)} + 
\sum_{{i \neq r} \atop {\ell(wt_{rs}\ov{t}_{ir}) = 
\ell(w)}} I_{wt_{rs}\ov{t}_{ir}}^{(k)}
\]
where $r>k$ is the last positive descent of $w$ and $s$ is maximal such
that $w_s < w_r$. 

For any $w\in \wt{W}_\infty$ which is increasing up to $k$, we
construct the $k$-transition tree $\wt{T}^k(w)$ with nodes given by
elements of $\wt{W}_\infty$ and root $w$ as in \S \ref{tes}. Let $r$
be the last descent of $w$. If $w=1$, or $k\neq 1$ and $r=k$, or $k=1$
and $r\in \{0,1\}$, then set $\wt{T}^k(w)=\{w\}$. Otherwise, let $s
= \max(i>r\ |\ w_i < w_r)$ and $\wt{\Phi}(w)= \wt{\Phi}_1(w)\cup
\wt{\Phi}_2(w)$, where
\begin{gather*}
\wt{\Phi}_1(w)= \{wt_{rs}t_{ir}\ |\ 1\leq i < r \ \ \mathrm{and} \ \ 
\ell(wt_{rs}t_{ir}) = \ell(w) \}, \\
\wt{\Phi}_2(w)= 
\{wt_{rs}\ov{t}_{ir}\ |\ i\neq r \ \ \mathrm{and} \ \
\ell(wt_{rs}\ov{t}_{ir}) = \ell(w) \}.
\end{gather*}
To define $\wt{T}^k(w)$, we join $w$ by an edge to each $v\in
\wt{\Phi}(w)$, and attach to each $v\in \wt{\Phi}(w)$ its tree
$\wt{T}^k(v)$.

The assertions of Lemma \ref{Tlemma} remain true for $\wt{T}^k(w)$,
with similar proof. When $k=0$, this is contained in \cite[Thm.\
4]{B}. For the case when $r>k>1$, $\wt{\Phi}_1(w)=\emptyset$, and
$w_s>0$, one observes that $wt_{rs}\ov{t}_{1r}\in \wt{\Phi}_2(w)$. We
deduce that for any $w\in \wt{W}_\infty$ which is increasing up to $k$,

\begin{equation}
\label{DStan}
\DS_w^{(k)} = I_w^{(k)} = \sum_{\la\, :\, |\la| = \ell(w)}
d^w_{\la}\,H_{\la}
\end{equation}
where the sum is over typed $k$-strict partitions $\la$ and $d^w_\la$
denotes the number of leaves of $\wt{T}^k(w)$ of shape $\la$. Moreover, 
for any such $w$, we have
\begin{equation}
\label{basicD}
\DS_w(X\,;Y) = \sum_{u(1_k\times v) = w}
I^{(k)}_u(X\,;Y)\AS_v(Y_{>k})
\end{equation}
where the sum is over all reduced factorizations $u(1_k\times v) = w$ with
$v\in S_\infty$.

Using equations (\ref{DStan}) and (\ref{basicD}), we obtain splitting
theorems for the single and double type D Schubert polynomials
$\DS_w$, as in \S \ref{splitC}. Fix two sequences $\fraka \, :\, a_1 <
\cdots < a_p$ and $\frakb \, :\, 0=b_1 < \cdots <b_q$ of nonnegative
integers and set $Y_i = \{y_{a_{i-1}+1},\ldots,y_{a_i}\}$ and $Z_j=
\{z_{b_{j-1}+1},\ldots,z_{b_j}\}$ for each $i,j$. We say that an element
$w\in \wt{W}_\infty$ is {\em compatible} with $\fraka$ if all
descent positions of $w$ are listed among $0, a_1,\ldots, a_p$, if
$a_1=1$, or contained in $\fraka$, otherwise.

\begin{thm}
\label{DSHG}
Suppose that $w\in \wt{W}_\infty$ is compatible with the sequence
$\fraka$. Then we have
\[
\DS_w(X\,;Y) = \sum_{u_1\cdots u_p = w} I_{u_1}(X\,;Y_1)G_{u_2}(Y_2)
\cdots G_{u_p}(Y_p)
\]
summed over all reduced factorizations $u_1\cdots u_p = w$ compatible
with $\fraka$ such that $u_2,\ldots,u_p\in S_\infty$.
\end{thm}

Given any sequence of partitions
$\underline{\la}=(\la^1,\ldots,\la^{p+q-1})$ with $\la^q$ $a_1$-strict
and typed, and $w\in \wt{W}_\infty$ such that $w$ and $w^{-1}$ are
compatible with the sequences $\fraka$ and $\frakb$, respectively, we
define
\[
g^w_{\underline{\la}} = \sum_{u_1\cdots u_{p+q-1} = w}
c_{\la^1}^{u_1}\cdots c_{\la^{q-1}}^{u_{q-1}}
d_{\la^q}^{u_q}c_{\la^{q+1}}^{u_{q+1}}\cdots c_{\la^{p+q-1}}^{u_{p+q-1}}.
\]
Here the sum is over reduced factorizations $u_1\cdots u_{p+q-1} = w$
compatible with $\fraka$, $\frakb$ such that $u_i\in S_\infty$ for all
$i\neq q$, and the integers $c_{\la^i}^{u_i}$ and $d^{u_q}_{\la^q}$
are as in (\ref{Geq}) and (\ref{DStan}), respectively.

\begin{cor}
\label{dbleDSHG}
Suppose that $w$ and $w^{-1}$ are compatible with
the sequences $\fraka$ and $\frakb$, respectively. Then 
$\DS_w(X\,;Y,Z)$ is equal to 
\[
\sum_{u_1\cdots u_{p+q-1} = w} G_{u_1}(0/Z_q)\cdots
G_{u_{q-1}}(0/Z_2) I_{u_q}(X\,;Y_1) G_{u_{q+1}}(Y_2) \cdots
G_{u_{p+q-1}}(Y_p)
\]
summed over all reduced factorizations $u_1\cdots u_{p+q-1} = w$
compatible with $\fraka$, $\frakb$ such that $u_i\in S_\infty$ for
all $i\neq q$. Furthermore, we have
\begin{equation}
\label{dbDSsplitting}
\DS_w(X\,;Y,Z) = \sum_{\underline{\la}} 
g^w_{\underline{\la}}\,
s_{\la^1}(0/Z_q)\cdots
H_{\la^q}(X\,;Y_1)\cdots s_{\la^{p+q-1}}(Y_p)
\end{equation}
summed over all sequences of partitions 
$\underline{\la}=(\la^1,\ldots,\la^{p+q-1})$
with $\la^q$ $a_1$-strict and typed.
\end{cor}

In the same manner as for the symplectic groups, the above splitting
results imply Giambelli and degeneracy locus formulas for the
orthogonal groups. We use the Giambelli formula for even orthogonal
Grassmannians from \cite{BKT4} in (\ref{dbDSsplitting}).

Table \ref{schubtable} lists the Billey-Haiman Schubert polynomials
for the root systems of type $\text{C}_3$ and $\text{D}_3$ indexed by
the elements $w$ in the respective Weyl groups which are increasing up
to $1$. In each case, the polynomial is written as a positive sum of
$k=1$ theta and eta polynomials in the variables $(X, y_1)$ times
$s_j(y_2)$ for $j\in\{0,1\}$.  The primed eta polynomials $H'_{\la}$
are indexed by $1$-strict partitions $\la$ of type $2$.

{\small{
\begin{table}[t]
\caption{Schubert polynomials with $w$ increasing up to $1$}
\centering
\begin{tabular}{|c|c|c||c|c|c|} \hline
$w$ & word & $\CS_w(X\,;Y)$ & $w$ & word &
$\DS_w(X\,;Y)$ \\ \hline

$123$ &  & $1$ & 
$123$ &  & $1$ \\

$213$ & $1$ & $\Ti_1$ & 
$213$ & $1$ & $H_1$ \\

$132$ & $2$ & $\Ti_1 + y_2$ & 
$\ov{2}\ov{1}3$ & $0$ & $H'_1$ \\

$2\ov{1}3$ & $01$ & $\Ti_2$ & 
$132$ & $2$ & $H_1 + H'_1 + y_2$ \\

$312$ & $21$ & $\Ti_{(1,1)}$ & 
$\ov{1}\ov{2}3$ & $01$ & $H_2$ \\

$231$ & $12$ & $\Ti_2 + \Ti_1\, y_2$ & 
$312$ & $21$ & $H_{(1,1)}$ \\

$1\ov{2}3$ & $101$ & $\Ti_3$ &
$\ov{3}\ov{1}2$ & $20$ & $H'_{(1,1)}$ \\

$3\ov{1}2$ & $021$ & $\Ti_{(2,1)}$ &
$231$ & $12$ & $H_2 + H_1\, y_2$ \\

$321$ & $212$ & $\Ti_{(2,1)}+\Ti_{(1,1)}\, y_2$ &
$\ov{2}3\ov{1}$ & $02$ & $H_2 + H'_1 \, y_2$ \\

$23\ov{1}$ & $012$ & $\Ti_3+\Ti_2 \, y_2$ &
$\ov{1}\ov{3}2$ & $201$ & $H_3$ \\

$1\ov{3}2$ & $2101$ & $\Ti_4$ &
$3\ov{2}\ov{1}$ & $021$ & $H_{(2,1)}$ \\

$3\ov{2}1$ & $1021$ & $\Ti_{(3,1)}$ &
$\ov{3}\ov{2}1$ & $120$ & $H'_{(2,1)}$ \\

$13\ov{2}$ & $1012$ & $\Ti_4+\Ti_3 \, y_2$ &
$\ov{1}3\ov{2}$ & $012$ & $H_3 + H_2\, y_2$ \\

$32\ov{1}$ & $0212$ & 
$\Ti_4+\Ti_{(3,1)}+\Ti_{(2,1)}\, y_2$ &
$321$ & $212$ & $H_3 + H_{(2,1)} + H_{(1,1)}\, y_2$ \\

$2\ov{3}1$ & $21021$ & $\Ti_{(4,1)}$ &
$\ov{3}2\ov{1}$ & $202$ 
& $H_3 + H'_{(2,1)} + H'_{(1,1)}\, y_2$ \\

$3\ov{2}\ov{1}$ & $01021$ & $\Ti_{(3,2)}$ &
$2\ov{3}\ov{1}$ & $2021$ & $H_{(3,1)}$ \\

$12\ov{3}$ & $21012$ & $\Ti_5 + \Ti_4 \, y_2$ &
$\ov{2}\ov{3}1$ & $2120$ & $H'_{(3,1)}$ \\

$31\ov{2}$ & $10212$ & 
$\Ti_{(4,1)}+\Ti_{(3,1)} \, y_2$ &
$\ov{1}2\ov{3}$ & $2012$ & $H_4 + H_3\, y_2$ \\

$2\ov{3}\ov{1}$ & $021021$ & $\Ti_{(4,2)}$ &
$3\ov{1}\ov{2}$ & $0212$ 
& $H_{(3,1)}+ H_{(2,1)} \, y_2$ \\

$21\ov{3}$ & $210212$ & 
$\Ti_{(5,1)}+\Ti_{(4,1)} \, y_2$ &
$\ov{3}1\ov{2}$ & $1202$ 
& $H'_{(3,1)} + H'_{(2,1)}\, y_2$ \\

$3\ov{1}\ov{2}$ & $010212$ & 
$\Ti_{(4,2)} + \Ti_{(3,2)}\, y_2$ &
$1\ov{3}\ov{2}$ & $12021$ & $H_{(3,2)}$ \\

$1\ov{3}\ov{2}$ & $1021021$ & $\Ti_{(4,3)}$ &
$2\ov{1}\ov{3}$ & $20212$ 
& $H_{(4,1)} + H_{(3,1)}\, y_2$ \\

$2\ov{1}\ov{3}$ & $0210212$ & 
$\Ti_{(5,2)} + \Ti_{(4,2)}\, y_2$ &
$\ov{2}1\ov{3}$ & $21202$ 
& $H'_{(4,1)} + H'_{(3,1)}\, y_2$ \\

$1\ov{2}\ov{3}$ & $10210212$ & 
$\Ti_{(5,3)} +\Ti_{(4,3)}\, y_2 $ &
$1\ov{2}\ov{3}$ & $120212$ 
& $H_{(4,2)} + H_{(3,2)}\, y_2$ \\
\hline
\end{tabular}
\label{schubtable}
\end{table}}}

\end{document}